\documentclass[11pt]{article}
\usepackage[margin=1in]{geometry}
\usepackage[T1]{fontenc}

\usepackage{graphicx,xcolor,cite,mathtools,bbold}
\usepackage{amssymb,amsmath,amsthm,mathrsfs,paralist,bm,esint,setspace}
\RequirePackage[colorlinks,citecolor=blue,urlcolor=blue]{hyperref}

\usepackage{cases}

\usepackage{color}

\DeclareMathOperator{\Cov}{Cov}

\DeclareMathOperator{\Dim}{Dim_{_H}}

\newcommand{\N}{\mathbb{N}}
\newcommand{\Z}{\mathbb{Z}}

\newcommand{\R}{\mathbb{R}}

\newcommand{\T}{\mathbb{T}}

\newcommand{\lip}{\text{\rm Lip}}
\newcommand{\Lp}{\text{\rm L}}

\renewcommand{\P}{\mathrm{P}}
\newcommand{\E}{\mathrm{E}}

\renewcommand{\d}{{\rm d}}

\newcommand{\e}{{\rm e}}
\renewcommand{\geq}{\geqslant}
\renewcommand{\leq}{\leqslant}
\renewcommand{\ge}{\geqslant}
\renewcommand{\le}{\leqslant}

\author{Davar Khoshnevisan\\University of Utah
	\and Kunwoo Kim\\POSTECH
	\and Carl Mueller\\University of Rochester}
\title{\bf On the valleys of the stochastic heat equation%
        \thanks{%
	Research supported in part by  the National Science Foundation grant 
DMS-1855439 (DK), the National Research Foundation of Korea grants  2019R1A5A1028324 and 2020R1A2C4002077 (KK), and by Simons Foundation Collaboration Grant 513424 (CM).}}
\date{Version of November 5, 2022}

\newtheorem{stat}{Statement}[section]
\newtheorem{proposition}[stat]{Proposition}

\newtheorem{theorem}[stat]{Theorem}
\newtheorem{lemma}[stat]{Lemma}
\theoremstyle{definition}

\newtheorem{remark}[stat]{Remark}

\numberwithin{equation}{section}

\begin{document}
\maketitle

\begin{abstract} 
We consider a generalization of the parabolic Anderson model driven by
space-time white noise, also called 
the stochastic heat equation, on the real line.  High peaks of solutions 
have been extensively studied under the name of intermittency, but less is 
known about spatial regions between peaks, which may loosely refer to as valleys. 
We present two results about the valleys of the solution. 

Our first theorem provides information about the size of valleys and the supremum of the 
solution over  a valley.  More precisely, we show that the supremum of the solution 
over a valley vanishes as $t\to\infty$, and we establish an upper bound of
$\exp\{-\text{const}\cdot t^{1/3}\}$
for the rate of decay.  We  demonstrate also that the length of a valley grows at 
least as $\exp\{+\text{const}\cdot t^{1/3}\}$ as $t\to\infty$.  

Our second theorem asserts that  the length of the valleys are eventually
infinite when the initial data has subgaussian tails.
\end{abstract}



%%%%%%%%%%%%%%%%%%%%%%%%%%%%%%%%%%%%%%%%%%%%%%%%%%%%%%
%%%%%%%%%%%%%%%%%%%% Introduction %%%%%%%%%%%%%%%%%%%%
%%%%%%%%%%%%%%%%%%%%%%%%%%%%%%%%%%%%%%%%%%%%%%%%%%%%%%

\section{Introduction \& Main Result}

Our objects of study are stochastic heat equations driven by multiplicative space-time 
white noise, including  the parabolic Anderson model,
whose solution are known to exhibit \emph{intermittency}. Intuitively speaking, intermittency refers  to the 
property that  the solution  tends to develop tall peaks distributed over 
small regions --- these are the so-called
\emph{intermittent islands} --- and those islands are separated 
by large areas where  the solution  is small --- these are the so-called \emph{valleys} or voids. There is an 
extensive literature about the  peaks particularly when the driving noise model does not depend on the time
variable --- see K\"onig \cite{Konig16} and its extensive references, for example --- and many techniques 
have been developed for understanding the  peaks. In the present
context of space-time white noise, a 
macroscopic fractal analysis has been developed, in Khoshnevisan, Kim, and Xiao \cite{KKX1,KKX2}, 
which characterizes how tall peaks 
are distributed over small islands. In the case of the parabolic Anderson model 
for space-time noise, much more detailed  results have recently become available; see for example
Corwin and Ghosal \cite{CorwinGhosal},
Das and Ghosal \cite{DasGhosal}, and Das, Ghosal, and Lin \cite{DasGhosalLin},
together with their substantial combined references.

In contrast to this literature, the regions between peaks, 
which we  call \emph{valleys}, have received less attention.  Our goal in this 
paper is to study the width of the valleys, how they grow
over time, and to estimate the supremum of our solution over the valley that straddles a
given point (here, the origin).  

Now we describe our results in more detail.  Let 
$\xi:=\{\xi(t\,,x)\}_{t\ge0,x\in\R}$ denote a two-parameter white noise.
That is, $\xi$ is a generalized mean-zero Gaussian random field with
generalized covariance
\begin{equation*}
	\Cov\left[ \xi(t\,,x)~,~\xi(s\,,y) \right] 
	= \delta_0(t-s)\delta_0(x-y)\qquad\text{
	for all $s,t\ge0$ and $x,y\in\R$. }
\end{equation*}
Let  $\mathcal{F}=(\mathcal{F}_t)_{t\ge0}$ designate the filtration of the white noise
$\xi$; that is,  for every $t>0$, $\mathcal{F}_t$ denotes the $\sigma$-algebra
generated by all Wiener integrals of the form $\int_{(0,t)\times\R}\varphi\,\d\xi$
as $\varphi$ ranges over $L^2(\R_+\times\R)$. We  assume without incurring loss of generality
that the  filtration $\mathcal{F}$ satisfies the usual conditions.  

The function $\sigma:\R\to\R$ is assumed to be non-random and satisfies
\begin{equation}\label{LL}
	\sigma(0)=0
	\qquad\text{and}\qquad
	0<{\rm L}_\sigma \le \lip_\sigma<\infty,
\end{equation}
where
\[
	{\rm L}_\sigma := \inf_{a\in\R\setminus\{0\}}
	\left| \frac{\sigma(a)}{a}\right|
	\qquad\text{and}\qquad
	\lip_\sigma:=\sup_{\substack{a,b\in\R:\\a\neq b}}
	\left| \frac{\sigma(b)-\sigma(a)}{b-a}\right|.
\]
Note, in particular, that 
\begin{equation} \label{eq:lip-sigma-condition}
	{\rm L}_\sigma |a| \le |\sigma(a)| \le \lip_\sigma|a|
	\qquad\text{for every }a\in\R.
\end{equation}

We say that $u$ is a solution to the stochastic heat equation  if
\begin{equation} \label{eq:SHE}\begin{split}
	&\partial_tu(t\,, x) 
		= \tfrac12 \partial_x^2u(t\,, x) 
       		+  \sigma(u(t\,, x))\xi(t\,,x)\quad \text{for }t>0, x\in \R,\\
	&\text{subject to}\quad u(0\,, x) =u_0(x)\quad \text{for }x\in \R.
\end{split}\end{equation}
where  $u_0:\R\to\R_+$ is assumed to be continuous and bounded
and $0<\|u_0\|_{L^1(\R)}\le\infty$.

Because the solution $u$ is not expected to be differentiable in either of
its two variables, \eqref{eq:SHE} must be interpreted in the generalized 
sense. Therefore,  we follow the treatment of  Walsh \cite{Walsh} and regard the SPDE \eqref{eq:SHE} 
as shorthand for the random integral equation
\begin{equation} \label{eq:mild}
	u(t\,, x)=\int_{-\infty}^\infty p_t(x-y) u_0(y) \, \d y+\int_{(0,t)\times\R}p_{t-s}(x-y)
\sigma(u(s\,, y))\xi(\d s\,  \d y),
\end{equation}
solved pointwise for every nonrandom choice of $t>0$ and $x\in\R$. 
Here $(t\,,x)\mapsto p_t(x)$ represents the fundamental solution to the heat equation on $\R$; that is,
\begin{equation*}
	p_{t}(x)=\frac{1}{\sqrt{2\pi t}}\exp\left(-\frac{x^2}{2t}\right),
\end{equation*}
and $\{S_t\}_{t\ge0}$ denotes the heat semigroup, which acts on $f:\R\to\R$ as
\begin{equation}\label{semigroup}
	(S_t f)(x)=\int_{-\infty}^\infty p_t(x-y)f(y)\,\d y.
\end{equation}
Finally, the double integral in \eqref{eq:mild} is a white noise integral in
the sense of Walsh \cite{Walsh}.
We refer to this solution as the \emph{mild solution} to \eqref{eq:SHE}; see Walsh
({\it ibid.}).

The existence and uniqueness of a mild solution to \eqref{eq:mild} is well 
known; see 
Walsh \cite[Chapter 3]{Walsh} for similar statements, and Dalang \cite{Dalang99}
for the general theory. Based on this general theory, we conclude 
that there is unique mild solution that is continuous in the variables $(t\,,x)$. Moreover,
the method of Mueller \cite{Mueller1} shows that
\begin{equation}\label{>0}
	\P\{ u>0\text{ on }(0\,,\infty)\times\R\}=1.
\end{equation}

Our main theorem follows.  

\begin{theorem} \label{th:SHE}
	If $u$ solves \eqref{eq:SHE} subject to $u_0\equiv1$, then
	there exist non-random numbers $\Lambda_i=\Lambda_i(\Lp_\sigma\,,\lip_\sigma)>0$
	$[i=1,2]$ and  an a.s.-finite random variable $T>0$ such that
	\[
		\sup_{|x| < \exp(\Lambda_1 t^{1/3})}u(t\,, x) \le \exp\left(-\Lambda_2t^{1/3}\right)
		\qquad\text{for all $t>T$}.   
	\]
\end{theorem}

Barlow and Taylor \cite{BT1989,BT1992} introduced a notion of macroscopic
Hausdorff dimension of a subset $E$ of $\R^d$. We can appeal to their 
dimension in order to shed some light on the content
of Theorem \ref{th:SHE}. In order to do that, let us first 
define $\mathcal{V}_n=(-\e^n\,,\e^n]^d$,
$\mathcal{S}_0=\mathcal{V}_0$, and $\mathcal{S}_{n+1}=\mathcal{V}_{n+1}\setminus
\mathcal{V}_n$ for all $n\in\Z_+$, and refer to (for $x\in\R^d$ and $r>0$) 
$Q=[x_1\,,x_1+r)\times\cdots\times[x_d\,,x_d+r)$
as an upright box with southwest corner $x$ and sidelength $\text{side}(Q)=r$.
Let  $\nu_\rho^n(E) =\inf\sum_{i=1}^m ( \text{side}(Q_i)/\e^{n})^\rho,$
where the infimum is taken over all upright boxes $Q_,\ldots,Q_m$ of side $\ge 1$
that cover $E\cap \mathcal{S}_n$. The \emph{Barlow-Taylor macroscopic Hausdorff
dimension} of $E$ is defined as the quantity
\[
	\Dim E = \inf\left\{ \rho>0:\
	\sum_{n=1}^\infty \nu^n_\rho(E)<\infty\right\},
\]
where $\inf\varnothing=0$. 

Khoshnevisan, Kim, and Xiao \cite{KKX2} have shown that  the tall peaks of $u$  form 
complex macroscopic space-time multifractals in the sense that 
there exist nonrandom numbers $A>a>0$ and $b,\varepsilon>0$ such that
\[
	2-A\beta^{3/2}\le
	\Dim\left\{ (\e^{t/\vartheta}\,, x) \in \R_+\times \R : \, u(t \,,x) \geq  \e^{\beta t} \right\}
	\le 2-a\beta^{3/2}\qquad\text{a.s.,}
\]
for every $\beta>b$ and $\vartheta\in(0\,,\varepsilon\beta^{-3/2})$. 
 In the above, we are using the convention
that ``$\Dim E<0$'' means that $E$ is bounded.

Among other things, this fact and the definition of
the macroscopic Hausdorff dimension together imply 
that there a.s.\ exist tall peaks of height $\e^{\beta t}$ over an interval of 
size $\asymp\e^{ t/\vartheta}$ for some $\vartheta>0$ and $\beta>b>0$ on an unbounded set
of times $t\gg1$; more precisely,  
\begin{equation}\label{unbdd}
 \left\{ t>0:\, \sup_{|x|\leq \exp(1+[ t/\vartheta] )} u(t\,, x) \geq \e^{\beta t}\right\}
 \text{ is a.s.\ unbounded}.
\end{equation}
Consider a small but fixed number $h_0\in(0\,,1)$ and define for every $t>0$,
\[
	\mathscr{L}(t) := \sup\left\{ \ell>0:\ \sup_{|x|\le\ell} u(t\,,x) \le h_0\right\},
\]
where $\sup\varnothing=0$. We may think of the interval $(-\mathscr{L}(t)\,,\mathscr{L}(t))\subset\R$
as the \emph{valley at time $t$} that straddles the origin.
Note that this valley might be empty at some time $t>0$, in which case $\mathscr{L}(t)=0$.
Because
\[
	\left\{ t>0:\, \sup_{|x|\leq \exp(1+[t/\vartheta])} u(t\,, x) \geq \e^{\beta t}\right\}
	\subseteq
	\left\{ t>0:\, \sup_{|x|\leq \exp(1+[t/\vartheta])} u(t\,, x) > h_0\right\},
\]
it follows from \eqref{unbdd} that $\mathscr{L}(t) \le \exp(1+[t/\vartheta])$ for an unbounded set of times
$t\gg1$. Since $\vartheta\in(0\,,\varepsilon\beta^{-3/2})$ and $\beta>b>0$
are arbitrary, we learn from this endeavor that
\begin{equation}\label{L:UB}
	\liminf_{t\to\infty} t^{-1} \log \mathscr{L}(t) \le  b^{3/2}\varepsilon^{-1}
	<\infty\qquad\text{a.s.}
\end{equation}
This is the best known upper bound to date, but is likely not sharp.
In the case of the parabolic Anderson model [$\sigma(z)=z$]
Das and Tsai \cite{DasTsai} have developed much sharper large-deviations
estimates than those in \cite[Proposition 3.1]{KKX2}. It might be possible to combine 
the Das-Tsai estimates instead of Proposition 3.1 of Khoshnevisan, Kim, and Xiao \cite{KKX2},
together with the remaining arguments of \cite{KKX2}, in order
to improve this to prove that the constant $b$ can be chosen arbitrarily. 
If this were so, then it would imply that \eqref{L:UB} might hold for every $b>0$
and hence $\liminf_{t\to\infty} t^{-1}\log\mathscr{L}(t)=0$ a.s.\
when $\sigma(z)\equiv z$.

In any case, Theorem \ref{th:SHE} assures us of the following complementary result: 
\[
	\liminf_{t\to\infty} t^{-1/3}\log \mathscr{L}(t)\ge \Lambda_1^{1/3}>0
	\qquad\text{a.s.,}
\]
and moreover tells us the solution is 
$\le\exp(-\Lambda_2 t^{1/3})$ everywhere in that valley at all sufficiently large times. 

Recently, Ghosal and Yi \cite{GY} have shown  that, in the case of the parabolic Anderson model
[$\sigma(z)\equiv z$],
$\Dim\{ (t\,, x) \in \R_+\times \R \, ; \, u(t\, ,x) \leq \e^{-\alpha t} \}=2$
a.s.\ provided that $\alpha$ is sufficiently small.
Their result is not 
about the length of the valleys. Rather, it tells us that there are many 
points $(t\,, x)$ where the solution is exponentially small. 
Intermittency could in principle imply that the supremum of the solution over a 
valley is much larger than its smallest, or even typical, value over the same valley.  We currently do not
know whether or not this is true however.

The strategy for the proof of Theorem \ref{th:SHE} is as follows:  We first decompose  
the initial profile $u_0\equiv1$ as 
\[ u_0(x)= \sum_{i=-M}^{M-1} v^{(i)}_0(x) + v_0^{(M)}(x),\]
where $v^{(i)}_0$ and $v^{(M)}_0$  are continuous and non-negative 
functions such that the support of $v^{(i)}$ is in $[i-1\,, i+1]$
and the support of $v^{(M)}_0$ is in $\R\setminus(-M\,, M)$. 
We prove that the solution $u$ to \eqref{eq:SHE} with $u_0\equiv 1$ can in
turn be decomposed  as
\[ 
	u(t\,, x) = \sum_{i=-M}^{M-1} v^{(i)}(t\,, x) +  v^{(M)}(t\,, x),
\]
where $v^{(i)}$ and $v^{(M)}$ satisfy parabolic Anderson models driven by
certain worthy martingale measures --- see \S\ref{sec:partition} and especially
\eqref{eq:SHE2} --- and starting from respective initial functions  $v^{(i)}_0$ and $v^{(M)}_0$. 
Once we establish this,  we freeze the time variable $t$ and appeal to 
the preceding decomposition with $M:=M(t)= 2 R(t)$,
where 
\[
	R(t) =\exp\left( -\Lambda_1 t ^{1/3}\right)\qquad\text{for all $t>0$}.
\]
On one hand, since $v^{(M)}_0=0$ for $|x|\leq 2R(t)$, we are able 
to  show that $\sup_{|x| \leq R(t)} v^{(M)}(t\,, x)$ is extremely small with very high probability. On 
the other hand,  when $i<M$,  the initial profile of $v^{(i)}$
has a compact support, and we can use the following theorem 
in order to prove that the global supremum of $v^{(i)}(t)$ tends rapidly to zero as $t\to\infty$.

\begin{theorem} \label{th:SHE-cc}
	Let $v$ solve the SPDE  \eqref{eq:SHE2} below with a continuous 
	and non-negative initial function $v_0$  that satisfies 
	$\limsup_{|x|\to\infty} x^{-2} \log v_0(x) <0$, keeping in mind 
	the convention $\log0=-\infty<0$.
	Then, there exists a non-random
	number $\Lambda_3 = \Lambda_3(\Lp_\sigma\,,\lip_\sigma)>0$ 
	and an a.s.-finite random time $T$ such that
	\[
		\sup_{x\in \R}v(t\,, x) \le \exp\left(-\Lambda_3t^{1/3}\right)
		\qquad\text{for all $t>T$}.
	\]
\end{theorem} 
Lemma \ref{lem:u=v} below ensures that the SPDE \eqref{eq:SHE2} is a generalization of our original
SPDE \eqref{eq:SHE}. Therefore, Theorem \ref{th:SHE-cc} implies that if the initial
data of the SPDE \eqref{eq:SHE} is non-negative
and  has subgaussian tails, then the global supremum of
the solution to \eqref{eq:SHE} vanishes at least as rapidly as $\exp(-\Lambda_2t^{1/3})$
as $t\to\infty$. That is, a specialization of Theorem \ref{th:SHE-cc}
implies the second announced result in the Abstract of the paper: With probability one,
$\mathscr{L}(t)=\infty$
for all sufficiently large $t$.

In order to prove Theorem \ref{th:SHE-cc} we first show  that
$\sup_{x\in \R} v(t\,, x)$ can be controlled by the total mass 
$\|v(t)\|_{L^1(\R)}$ of $v$; this is done in \S\ref{sec:control}. One may see a similar 
result in our earlier paper \cite{KKM} where we consider a stochastic heat equation driven 
by space-time white noise on the one-dimensional torus $\R/\Z$ rather than on $\R$.  Because 
$\R$ is not compact, we need to make significant 
modifications to the method of \cite{KKM} especially when  we estimate  moments; see 
\S\ref{sec:moment}.  Once we  are able to prove that
$\sup_{x\in \R} v(t\,, x)$ can be controlled by $\|v(t)\|_{L^1(\R)}$, we
appeal to a known result 
about dissipation of the total mass of the solution (see Chen, Cranston, Khoshnevisan, and Kim
\cite{CCKK}) to prove 
Theorem \ref{th:SHE-cc}; this is done in \S\ref{sec:dissipation}. 
Finally, we combine the results from \S\S\ref{sec:partition}-\ref{sec:dissipation} 
in order to verify Theorem \ref{th:SHE} in \S\ref{sec:main-proof}.

We conclude the Introduction by setting forth some notation that will be
used throughout the paper.  In order to simplify some of the formulas, we distinguish between the spaces
$L^k$ and $L^k(\P)$ by writing the former as 
\begin{equation*}
	L^k := L^k[\R] \qquad [1\le k<\infty].
\end{equation*}
Thus, for example, if $f\in L^k[\R]$ for some $1\le k<\infty$, then
\begin{equation*}
	\|f\|_{L^k} := \left[ \int_{-\infty}^\infty |f(x)|^k\,\d x\right]^{1/k}.
\end{equation*}
We will abuse notation slightly and write
\begin{equation*}
	\|f\|_{L^\infty}:= \sup_{x\in \R}|f(x)|,
\end{equation*}
in place of the more customary essential supremum.

The $L^k(\P)$-norm
of a random variable $Z\in L^k(\P)$ is denoted by
\begin{equation*}
	\|Z\|_k := \left\{ \E\left(|Z|^k\right)\right\}^{1/k}
	\qquad\text{for all $1\le k<\infty$.}
\end{equation*}
On multiple occasions we refer to $C_b(\R)$
as the collection of bounded and continuous 
real-valued functions on $\R$, and to $C_b^+(\R)$ as the cone
of all nonnegative elements of $ C_b(\R)$. Finally,
we follow Shiga \cite{Shiga}
and define $C^+_{\textit{rap}}(\R)$ to be the set of functions in $C_b^+(\R)$ that satisfy 
the rapid decrease condition
$\limsup_{|x|\to\infty}|x|^{-1}\log v_0(x) =-\infty$.

\section{A partition of  the stochastic heat equation} \label{sec:partition}
Suppose we write the initial function $u_0 \in L^\infty$ as $u_0= v_0+w_0$ and call $v$ and $w$ the solutions to \eqref{eq:SHE} with respective
initial functions $v_0$ and $w_0$.  Since \eqref{eq:SHE} may not be linear, 
except when $\sigma(u)=u$, there is no reason to believe that $u= v + w$
in general. Instead we show in this section that $u=v+w$ where $v$ and $w$ 
solve the closely related stochastic heat equations \eqref{eq:SHE2} below with 
respective initial functions $v_0$ and $w_0$. 
In other words, we plan to show that, to a certain extent, the semilinear SPDE \eqref{eq:SHE}
always has a kind of ``linear dependence on the initial data.'' We will see later  that
this kind of linear dependence on initial data suffices for our needs thanks to 
condition \eqref{LL}.

To implement our splitting, we first rewrite \eqref{eq:SHE} so that it looks more 
like the linear parabolic Anderson equation; that is, we write
\begin{equation}  \label{eq:SHE0}\begin{split}
	&\partial_tu(t\,, x) = \tfrac12 \partial_x^2u(t\,, x) 
	       +  u(t\,,x)\tilde{\xi}(t\,,x) \quad \text{for }t>0, x\in \R, \\
	&\text{subject to}\quad u(0)=u_0\quad\text{on }\R,
\end{split}\end{equation}
where
\begin{equation} \label{eq:def-sigma-tilde}
	\tilde{\xi}(t\,,x) =\tilde{\xi}(t\,,x\,;u) := \tilde{\sigma}(t\,,x)\xi(t\,,x)
	\quad\text{for}\quad
	\tilde \sigma(t\,, x) = \tilde \sigma(t\,, x\,;u): = \frac{\sigma(u(t\,, x))}{u(t\,, x)}.
\end{equation}
Thanks to \eqref{>0}, the random function $\tilde{\sigma}$, and hence
the random distribution $\tilde\xi$, are well defined.

From  \eqref{eq:lip-sigma-condition} we may conclude that, with probability one,
\begin{equation*}
	{\rm L}_\sigma \le |\tilde{\sigma}(t\,,x)| \le \lip_\sigma
	\qquad\text{for every }(t\,,x)\in[0\,,\infty)\times\R.
\end{equation*}
Therefore, in the sense of Walsh \cite{Walsh}, we can regard the new noise
$\tilde{\xi}$ as a worthy martingale measure with a dominating measure 
that is bounded below and above by constant multiples of Lebesgue measure.   More precisely,
if
\[ 
	M_t(\varphi):= \int_{(0, t)\times \R} \varphi(s\,, y)  \, \tilde\xi(\d s \, \d x) = 
	\int_{(0, t)\times\R} \varphi(s\,, y) \frac{\sigma(u(s\,, y))}{u(s\,, y)} \, \xi(\d s\, \d y),
\] 
then  \eqref{eq:lip-sigma-condition} implies that  
for all $t>0$ and for all nonnegative $\varphi,\psi\in C_c(\R_+\times\R)$,
\begin{equation}\label{eq:worthy}
	{\rm L}^2_\sigma \int_0^t \int_{-\infty}^\infty \varphi(s\,, y) \psi(s\,, y)  \, \d y \, \d s 
	\leq \langle M(\varphi)\,, M(\psi) \rangle_t \leq \lip_{\sigma}^2  
	\int_0^t \int_{-\infty}^\infty \varphi(s\,, y) \psi(s\,, y) \, \d y \, \d s. 
\end{equation} 
We emphasize  that the upper and lower bounds on $\langle M(\varphi), M(\psi) \rangle_t$ 
are not random and, in particular, do not depend on $u$. 

Now choose and fix some $v_0\in C^+_b(\R)$ and
consider solutions $v$ to the following parabolic Anderson model forced by
the martingale measure $\tilde{\xi}$:
\begin{equation} \label{eq:SHE2}
\begin{split}
	&\partial_t v(t\,, x) 
		= \tfrac12 \partial_x^2 v(t\,, x) 
	       +  v(t\,,x)\, \tilde{\xi}(t\,,x)\quad \text{for }t>0,  x\in \R, \\
	&\text{subject to}\quad v(0) =v_0 \quad \text{on }\R. 
\end{split}\end{equation}
Similar to what was done \eqref{eq:mild}, we can define a mild solution to \eqref{eq:SHE2} as 
\begin{equation} \label{eq:mild-v}
v(t\,, x)=\int_{-\infty}^\infty p_t(x-y) v_0(y) \, \d y+\int_{(0,t)\times\R}p_{t-s}(x-y)
v(s\,, y) \, \tilde \xi(\d s\,  \d y).
\end{equation}
As was mentioned before, $\tilde\xi$ is a worthy  martingale measure thanks to \eqref{eq:worthy},  
and therefore the stochastic integral in \eqref{eq:mild-v} can be understood   in the sense of Walsh \cite{Walsh}. In addition, since \eqref{eq:SHE2} is linear in $v$, we may use  Walsh's theory  \cite{Walsh}
--- see also   Shiga \cite[Theorems  2.2 and 2.3]{Shiga} ---
in order to  conclude that \eqref{eq:SHE2} has a unique mild solution $v$ 
that is  a.s.\ non-negative and  continuous on $[0\,,\infty)\times\R$. Therefore,
the uniqueness theorem for such SPDEs implies the following.

\begin{lemma}\label{lem:u=v}
	If $u$ denotes the  solution to \eqref{eq:SHE} 
	with the initial function $v_0 \in C^+_b(\R)$  and $v$
	denotes the solution to \eqref{eq:SHE2} with the same initial function $v_0$, then
	$u=v$ almost surely.
\end{lemma} 

Before we move on, let us pause to summarize the philosophy of the construction of this section up to this point:
$u$ solves the original SPDE \eqref{eq:SHE} starting from non-negative $u_0\in L^\infty$. With $u$ fixed in our minds,
we may solve \eqref{eq:SHE2} for every $v_0\in C^+_b(\R)$.  Lemma \ref{lem:u=v} assures us
that $v=u$ if $v_0=u_0$. However,  it should be clear also that
$u$ and $v$ can differ when $u_0\neq v_0$.
%
%\begin{proof}
%We first use  the mild formulation of $u$ and $v$ to get 
%\[ u(t\,, x)-v(t, x)=\int_{(0,t)\times \R} p_{t-s}(y-x) \frac{\sigma(u(s,  y))}{u(s\,, y)} \left(u(s\,, y) - v(s, y) \right) \xi(\d s \, \d y).\]
%Our conditions directly imply that, for all $(t,x)\in[0,\infty)\times\R$,  we 
%have $E\left[|u(t,x)-v(t,x)|^2\right]<\infty$.  
%Now Walsh's extension of Ito's isometry and \eqref{eq:lip-sigma-condition} imply that 
%\begin{equation} \label{eq:diff:u-v}
%\begin{aligned} 
%\E\big[|u(t\,, x)&-v(t, x)|^2\big] \\
%&=\E \int_0^t \d s  \int_{-\infty}^\infty \d y\ p_{t-s}^2(y-x) \frac{\sigma^2(u(s\,, y))}{u^2(s, y)} \left|u(s\,, y) - v(s, y) \right|^2 \\
%&\leq \lip_\sigma^2 \int_0^t \d s   \int_{-\infty}^\infty \d y \  p_{t-s}^2(y-x) \E  \left|u(s\,, y) - v(s, y) \right|^2.
%\end{aligned}
%\end{equation}
%Define 
%$\mathcal{E}(\alpha):=\sup_{t\geq 0}\sup_{x\in \R} \e^{-\alpha t}\E|u(t\,, x)-v(t, x)|^2$
% for $\alpha>0$. Arguing as in Gronwall's inequality, we see that \eqref{eq:diff:u-v} says that  
%\[ \mathcal{E}(\alpha) \leq \mathcal{E}(\alpha)\,  \lip_\sigma^2 \int_0^\infty \e^{-\alpha s} p_{2s}(0) \d s  = \frac{\mathcal{E}(\alpha)\,  \lip_\sigma^2}{2\sqrt{\alpha}}.\] 
%By taking $\alpha$ large enough, we get 
%\[ 
%\sup_{t\geq 0}\sup_{x\in \R} \E|u_t(x)-v_t(x)|^2 = 0,
%\]
%and thus the continuity of $u$ and $v$ implies the given claim. 
%\end{proof}
%
The following remark describes how we intend to use this observation
in conjunction with Lemma \ref{lem:u=v}.

\begin{remark}\label{rem:partition}
	Since \eqref{eq:SHE2} is linear in $v$,  we can partition the solution 
	$u$ to \eqref{eq:SHE} with the initial function $u_0\equiv1$. 
	Indeed, a partition of unity enables to  write $u_0\equiv1$ as 
	\[ 
		u_0= \sum_{i=-M}^{M-1} v^{(i)}_0 + v_0^{(M)},
	\]
	where $v^{(i)}_0,v^{(M)}_0:\R\to\R_+$  are  continuous and non-negative functions 
	such that $v^{(i)}_0$ is supported in $[i-1\,, i+1]$ for $i=-M,\ldots,M-1$, and 
	$v^{(M)}_0$ is supported in $\R\setminus(-M\,, M)$.  
	Because the SPDE \eqref{eq:SHE2} is linear, we may superimpose solutions and
	appeal to Lemma \ref{lem:u=v} in order to decompose $u$ as follows:
	\[ 
		u(t\,, x) = \sum_{i=-M}^{M-1} v^{(i)}(t\,, x) +  v^{(M)}(t\,, x)
		\qquad\text{for all $t\ge0$ and $x\in\R$},
	\]
	where $v^{(-M)},\ldots, v^{(M)}$ satisfy  the SPDE \eqref{eq:SHE2} 
	with respective initial functions  $v^{(-M)}_0,\ldots, v^{(M)}_0$. 
\end{remark}
Let us conclude this section with a remark about the strong Markov property
(henceforth, denoted by SMP).
Let $u$ denote the solution to \eqref{eq:SHE} subject to $u_0\in L^\infty$.
It is well known that $\{u(t)\}_{t\ge0}$ is a diffusion with values in the
space $C(\R)$. In order to write down exactly
what this means, we need to first introduce some measure-theoretic notation:
For every $t\ge0$ and $x\in\R$ define
\begin{equation}\label{omega}
	\omega(t\,,x) =  \int_{(0,t)\times(-x_-,x_+)}\d\xi,
\end{equation}
where $x_-=-\min(0\,,x)$ and $x_+=\max(0\,,x)$, and the integral is defined in the sense
of Wiener. 
Elementary properties of
the Wiener integral show that $\omega$ is a Brownian sheet indexed by $\R_+\times\R$
and, as such, $\omega\in C(\R_+\times\R)$ a.s.; see Walsh \cite[Chapter 1]{Walsh}.
Moreover, the Brownian filtration
$\mathcal{F}$ is nothing but the filtration
generated by the infinite-dimensional Brownian motion $\{\omega(t)\}_{t\ge0}$.

We use a standard relabeling from measure
theory in order to be able to assume, without any loss in generality, that the underlying
probability space is $\Omega=C(\R_+\times\R)$ on which $\omega$ acts as a coordinate
function. In this way, every random variable on $\Omega$ is a Borel function of the coordinate 
functions $\omega$. We omit the remaining measure-theoretic details. Instead
we observe that the distribution-valued random variable $\xi$ is therefore also
a function of $\omega$, as is shown in \eqref{omega}. 

We may define a shift operator $\theta(t):\Omega\to\Omega$ for every $t\ge0$ as follows:
$(\theta(t)\omega)(s\,,x)=\omega(s+t\,,x)$ for all $s\ge0$ and $x\in\R$.
This induces a shift $\theta(t)X$ on every random variable $X$ via
$\theta(t) X (\omega)=X(\theta(t)\omega)$. If $\tau$ is a stopping time with respect 
to the filtration $\mathcal{F}$, then the random shift $\theta\circ\tau$ is well defined:
We simply define $(\theta\circ\tau)(t)(\omega) = \theta(\tau(t)(\omega))$
for every $t\ge0$ and $\omega\in\Omega$. We might write $\theta(\tau)$ in place
of $\theta\circ\tau$.

It is not hard to check that if $\tau$ is a finite stopping time with respect to
$\mathcal{F}$, then $\theta(\tau)\xi$ is a copy of $\xi$ that is independent of $\mathcal{F}_{\tau}$, where
the latter $\sigma$-algebra is defined in the usual sense.
The SMP of $u$ can now be cast as the slightly stronger assertion that
the space-time random field $\theta(\tau)u$ solves \eqref{eq:SHE}
where the space-time white noise $\xi$ is replaced by the space-time white noise $\theta(\tau)\xi$.

Let now $v_0\in C^+_b(\R)$ be nonrandom, and define $v$ to be the solution to
\eqref{eq:SHE2} starting from $v_0$. Because $\tilde{\sigma}$ in \eqref{eq:def-sigma-tilde}
is random, more specifically it is a mapping from $\Omega$ to $\Omega$, the process $\{v(t)\}_{t\ge0}$ does not
satisfy the SMP (though it is an adapted random field). We have introduced the measure-theoretic
notation above in order to discuss how the lack of the SMP of $\{v(t)\}_{t\ge0}$ can be mostly salvaged.

Direct inspection  leads to the following whose proof is omitted as it
follows well-known argument; see Da 
Prato and Zabczyk \cite[\S9.2]{DZ14}.

\begin{lemma}\label{lem:SMP}
	Let $v$ denote the solution to \eqref{eq:SHE2} starting from a non-random
	$v_0\in C^+_b(\R)$. Then, $\{(u(t)\,,v(t))\}_{t\ge0}$ is a diffusion
	with values in the space $C(\R\,,\R^2)$.
	Moreover, for every finite stopping time $\tau$,
	$\theta(\tau)v$ solves \eqref{eq:SHE2} with $(v_0\,,\tilde{\xi})$ replaced by
	$(\theta(\tau)v\,,\theta(\tau)\tilde{\xi})$ and the underlying Brownian filtration
	replaced by $\mathcal{F}_{\tau+\bullet}$.
\end{lemma}

In order to see how this lemma salvages a portion of
the SMP of the process $t\mapsto v(t)$, 
let us define for all non-random functions $\varphi\in C_c(\R_+\times\R)$ and all $t>0$,
\[
	\tilde{M}_t(\varphi) = \int_{(0,t)\times\R} \varphi(s\,,y)\,\tilde{\xi}(\d s\,\d y)
	=\int_{(0,t)\times\R} \frac{\varphi(s\,,y)\sigma(u(s\,,y))}{u(s\,,y)}\,\xi(\d s\,\d y),
\]
the second identity being a consequence of the definition in \eqref{eq:def-sigma-tilde}.
The basic properties of the Walsh stochastic integral ensure that $\tilde{M}$ defines
a worthy martingale measure whose dominating measure satisfies \eqref{eq:worthy}
with exactly the same constants as does $M$.
The basic properties of Walsh stochastic integrals show that $\theta(\tau)\tilde{M}$
is the martingale measure that correspond to the noise $\theta(\tau)\tilde{\xi}$ that
arose in Lemma \ref{lem:SMP}. The martingale measure $\theta(\tau)\tilde{M}$
is worthy and satisfies \eqref{eq:worthy} as well, also with exactly the same
constants as does $M$. Because our work with $\tilde\xi$ does not involve  
knowing the law of $\tilde\xi$, rather its property
\eqref{eq:worthy} only, it follows that many of the properties of $\tilde\xi$ that we study here
are (typically ``up to constants'') the same as those properties for $\theta(\tau)\tilde\xi$.
And therefore the same can be said of $v$ and $\theta(\tau)v$: They do not always have the
same law (if this were the case, then this would be the SMP), rather 
they have the same properties (typically ``up to constants''). 

\section{Control of tall peaks by total mass} \label{sec:control}

In  this section, we consider  the  solution $v$ to the stochastic heat equation \eqref{eq:SHE2}
that is driven by the worthy martingale measure $\tilde\xi$.  
We write the mild formulation for \eqref{eq:SHE2} in the same manner as in  Walsh \cite{Walsh}. Namely, for all $t>0$ and $x\in\R$,
\begin{equation}\label{mild}
	v(t\,,x) = (S_t v_0)(x) +\mathcal{I}(t\,,x) ,
\end{equation}  
where $\{S_t\}_{t\ge0}$ designates the heat semigroup --- see \eqref{semigroup} ---  and 
\begin{equation}\label{eq:noise}
	\mathcal{I}(t\,,x)  :  =\int_{(0,t)\times\R}
	p_{t-s}(x\,,y) v(s\, , y)   \, \tilde \xi(\d s\,\d y).
\end{equation}

The principal aim of this section  is to prove  the following proposition 
which basically says that the tallest peak height $\|v(t)\|_{L^\infty}$ of $v$ at time $t$
can be controlled by the total mass $\|v(t)\|_{L^1}$ of $v$;
this fact will play a role in the proof of Theorem \ref{th:SHE-cc}. 
It might help to recall from Introduction that $C^+_{\textit{rap}}(\R)$ denotes the set of all
functions in $C_b^+(\R)$ that decay at least exponentially rapidly at 
$\pm\infty$.

\begin{proposition}\label{prop:WLOG}
	Assume that $v_0 \in C^+_{\textit{rap}}(\R)$.  For every 
	$\gamma \in \left(\frac43\,, 2\right]$ and  $\beta \geq 6/(3\gamma-4)$
	there exist numbers $c_1,c_2>0$ --- that only depend on $\gamma, \beta, \lip_\sigma$ ---  
	such that
	\begin{equation}\label{cond:WLOG}
		\P\left\{ \sup_{0\le t\le n}\frac{\|v(t)\|_{L^\infty}}{\|v(t)\|_{L^1}}
		\geq n^\beta \right\} < c_1 \exp( - c_2 n)
		\qquad\text{for every $n\ge1$}.
	\end{equation}
\end{proposition}  

A similar result can be found  in our earlier paper \cite[Theorem 3.1]{KKM}, valid
in the case that the spatial domain is the torus $\R/\Z$ instead of $\R$.\footnote{Actually, Theorem
3.1 of \cite{KKM} is about SPDEs over $\R/(2\Z)$, but it is clear from the proof
that any other torus would also work.}
Although the proof of Proposition \ref{prop:WLOG} borrows  liberally from the ideas of
our earlier paper ({\it ibid.}), there also are several significant differences. Perhaps the first obvious 
difference is that the spatial domain is now $\R$, which is not compact. The
change from $\R/\Z$ to $\R$ requires  making several non-trivial modifications to our
earlier arguments, especially when we estimate the moments of the solution.  Those modifications
involve ``factorization'' ideas from semigroup theory; see Da Prato, 
Kwapie\'{n}, and Zabczyk \cite{DPKZ87}, and in particular, 
Cerrai \cite{Cerrai} and Salins \cite{Salins}.

There is another  difference between the proof of Proposition \ref{prop:WLOG}
and the earlier methods of \cite{KKM}. Namely, the proof of Theorem 3.1 of \cite{KKM} 
hinged on the SMP of  the solution $u$ to \eqref{eq:SHE}.  In the following, those arguments will be applied to $v$
by appealing to Lemma \ref{lem:SMP}.
Among other things, if $\tau$ is a stopping time for the Brownian filtration $\mathcal{F}$, 
then we condition on $\mathcal{F}_{\tau}$ to see that
$\tilde{v}=v(\tau(\omega)+\bullet)$ solves the parabolic Anderson model, 
\begin{equation} \label{eq:tilde v} 
\begin{aligned}
	& \partial_t \tilde v  = \tfrac{1}{2} \partial_x^2 \tilde v
		+\tilde v\theta(\tau)\tilde \xi, \\
	&\text{subject to  }\tilde v(0)= v(\tau).
 \end{aligned}
 \end{equation}
The point being that
that $\theta(\tau)\tilde\xi$  defines a worthy martingale measure 
with a  dominating measure that is bounded below and above by the same
constant multiples of Lebesgue measure as did $\tilde\xi$; see \eqref{eq:worthy}.

\subsection{Moment estimates} \label{sec:moment}
  We first estimate the moments of the solution $v(t\,,x)$   to \eqref{eq:SHE2}
  at a fixed point $(t\,,x)\in(0\,,\infty)\times\R$. It might help to recall that
  $\{S_t\}_{t\ge0}$ denotes the heat semigroup; see \eqref{semigroup}.

\begin{lemma}\label{lem:moment} 
There exists a real number  $A:=A(\lip_\sigma)>0$ such that 
\begin{equation}
	\|v(t\,, x)\|_k^2 \leq A \|v_0\|_{L^\infty} (S_tv_0)(x) \exp\left(A k^2 t\right),
\end{equation}
uniformly for all $t>0$, $x\in \R$, all nonnegative functions $v_0\in L^\infty$, and $k\ge2$. 
\end{lemma}

\begin{proof}
	We develop  some of the ideas of Foondun and Khoshnevisan \cite{FK}.  
	Let us consider Picard iteration: Define a sequence $\{ v^{(n)}\}_{n\geq 0}$ as 
	\[
		v^{(0)}(t\,, x):= (S_t v_0)(x)
		\quad\text{and}\quad
		v^{(n)}(t\,, x) := (S_t v_0)(x) + \mathcal{I}^{(n)}(t\,, x),
	\]
	for all $t>0$, $x\in\R$, and $n\in\N$, where 
	\[
		\mathcal{I}^{(n)}(t\,, x):= \int_{(0,t)\times\mathbb{\R}} 
		p_{t-s}(x\,,y)\, v^{(n-1)}(s\,,y)\, \tilde \xi(\d s\,\d y).
	\]
	The random field $v^{(n)}$ is the $n$th-stage 
	Picard-iteration approximation of $v$. Next we follow Walsh 
	\cite[Ch.\ 3]{Walsh} and obtain the following: For every $T>0$
	and $k\in[2\,,\infty)$,
	\begin{equation}\label{Picard}
		\lim_{n\to\infty}\adjustlimits\sup_{t\in(0,T]}\sup_{x\in\mathbb{T}}
		\E\left( \left| v^{(n)}(t\,, x) - v(t\,,x) \right|^k \right)=0.
	\end{equation}
	We appeal to a Burkholder-Davis-Gundy type inequality for stochastic 
	convolutions \cite[Proposition 4.4, p.\ 36]{cbms} and the worthy condition
	\eqref{eq:worthy} on the underlying noise $\tilde\xi$ in order to see that 
	\begin{equation}\label{role of noise}
		\left\| v^{(n)}(t\,, x)   \right\|_k^2 \leq 2 \left|(S_t v_0)(x) \right|^2 +  
		8k \lip_\sigma^2 \int_0^t \d s \int_{-\infty}^\infty \d y \, 
		[p_{t-s} (y-x)]^2  \left\| v^{(n-1)}(s\,, y)   \right\|_k^2. 
	\end{equation}
	It might help to pause and observe that $(S_t v_0)(x)>0$ for all $t>0$ and $x\in \R$.
	Let $\beta>0$ be a fixed parameter, and divide both sides of the preceding display 
	by $\exp(\beta t)\, (S_t v_0)(x)$ in order to find that
	\begin{align*}
		&\frac{\e^{-\beta t}\left\| v^{(n)}(t\,, x)   \right\|_k^2 }{(S_t v_0)(x)} \\ 
		&\leq 2 (S_t v_0)(x) + \frac{8k\lip_\sigma^2 }{(S_t  v_0)(x)} \int_0^t \d s \, 
			\e^{-\beta(t-s)}\int_{R} \d y \,  [p_{t-s} (y-x)]^2 (S_s v_0)(y)\,  
			\frac{\e^{-\beta s}   \left\| v^{(n-1)}(s\,, y)   \right\|_k^2}{(S_s v_0)(y)} \\
		& \leq 2\|v_0\|_{L^\infty} + \frac{8k\lip_\sigma^2}{(S_t  v_0)(x)} 
			\int_0^t \d s \,  \frac{\e^{-\beta(t-s)}}{\sqrt{2\pi (t-s)}} 
			\int_{-\infty}^\infty \d y \,  p_{t-s}(y-x)  (S_s v_0)(y)\,  \frac{\e^{-\beta s}
			\left\| v^{(n-1)}(s\,, y)   \right\|_k^2}{(S_s v_0)(y)}. 
	\end{align*}
	In the last inequality above, we use the fact that 
	$|S_t v_0(x)| \leq \|v_0\|_{L^\infty}$ and $p_t(z) \leq (2\pi t)^{-1/2}$ for all $z\in \R$. Define 
	extended real numbers $\Phi_0,\Phi_1,\ldots$ as
	\[
		\Phi_n: = \sup_{t> 0} \sup_{x\in \R} \frac{\e^{-\beta t}\left\| v^{(n)}(t\,, x)
		\right\|_k^2 }{(S_t v_0)(x)} \quad \text{for every $n\in\Z_+$}.
	\]
	The semigroup property of $\{S_t\}_{t\ge0}$ implies that,  for all 
	$n\in\N$, $t>0$, and $x\in \R$, 
	\begin{align*} 
		\frac{\e^{-\beta t}\left\| v^{(n)}(t\,, x)   \right\|_k^2 }{(S_t v_0)(x)} 
			&\leq 2\|v_0\|_{L^\infty} +  \frac{8k\lip_\sigma^2}{(S_t  v_0)(x)}\, 
			\Phi_{n-1} \int_0^t \d s \,  \frac{\e^{-\beta s}}{\sqrt{2\pi s}}
			\int_{-\infty}^\infty \d y \,  p_{t-s}(y-x)  (S_s v_0)(y) \\
		&=2\|v_0\|_{L^\infty} + 8k\lip_\sigma^2\, \Phi_{n-1} 
			\int_0^\infty \d s \,  \frac{\e^{-\beta s}}{\sqrt{2\pi s}} \\
		&=2\|v_0\|_{L^\infty} + \frac{8k\lip_\sigma^2  }{\sqrt{2\beta}} \,\, \Phi_{n-1}. 
	\end{align*}
	Since $v^{(0)}(t\,, x)= (S_t v_0)(x)$,  we have
	$\Phi_0 \leq \|v_0\|_{L^\infty}$. This  implies that all of the
	$\Phi_n$s are finite.   In addition, if we choose $\beta:=A k^2$
	for some constant $A:=A(\lip_\sigma)$, then we apply induction to the preceding
	in order find that
	\[
		\Phi_n  \leq 2\|v_0\|_{L^\infty} + \tfrac{1}{2} \Phi_{n-1}
		\qquad\text{for every $n\in\N$.}
	\]
	This implies that  $\Phi_n \leq  4\|v_0\|_{L^\infty}$ for all $n\in\Z_+$.
	Therefore, \eqref{Picard} and Fatou's lemma together yield
	\[  
		\adjustlimits\sup_{t\geq 0} \sup_{x\in \R} \frac{\e^{-Ak^2 t}
		\| v(t\,, x) \|_k^2 }{(S_t v_0)(x)} \leq 4\|v_0\|_{L^\infty},
	\]
	and hence complete the proof. 
\end{proof}

 We now  have  the following moment estimate of the noise term [see \eqref{eq:noise}]: 

\begin{lemma}\label{lem:tp2}
	For every  $0<\theta<\frac14$
	there exists a number $c=c(\theta\,,\lip_\sigma)>0$ such that 
	\begin{equation}\label{eq:moment_noise}
		\E\left(  \sup_{s\in[0,t]}\| \mathcal{I}(s) \|_{L^\infty}^k\right)
		\le (ck)^{k/2}\exp(ck^3t)
		\|v_0\|_{L^\infty}^{k-1} \|v_0\|_{L^1} t^{k\theta}, 
	\end{equation}
	uniformly for all $t\in(0\,,1]$,  nonnegative functions $v_0\in L^1\cap L^\infty$, 
	and real numbers $k\in[2\,,\infty)$ that satisfy $k(1-4\theta)>2$. 
\end{lemma}

\begin{proof}
	The proof of Lemma \ref{lem:tp2} is similar to results in Cerrai \cite{Cerrai} and 
	Salins \cite[Appendix A]{Salins} that consider SPDE on bounded domains. In order
	to adapt to the present setting where the spatial domain is
	$\R$, we modify the ideas of \cite{Cerrai,Salins} and also keep track of the dependence 
	of constants on $k$ and $t$.  
	
	First, we use the factorization method of Da Prato, 
	Kwapie\'{n}, and Zabczyk \cite{DPKZ87} to write 
	\begin{align*}
	 	\mathcal{I}(t\,, x) &=  \int_{(0,t)\times\R}
			p_{t-s}(x\,,y)\,  v(s\,,y))\,\tilde \xi(\d s\,\d y) \\
		&=\frac{\sin(\alpha \pi)}{\pi} \int_0^t (t-s)^{\alpha-1} 
			\left(S_{t-s} [\mathcal{I}_\alpha (s)]\right)(x) \, \d s,
	\end{align*}
	where $\alpha \in (0\,, 1)$ is arbitrary but fixed, and 
	\[
		\mathcal{I}_\alpha (t\,, x):=
		\int_{(0, t)\times \R}  (t-s)^{-\alpha} p_{t-s}(y-x)  v(s\,, y) \, \tilde \xi(\d s \, \d y).
	\]
	In this way, we find that
	\begin{equation}\label{eq:bound1} 
		\| \mathcal{I} (t)\|_{L^\infty} \leq \frac{1}{\pi} \int_0^t (t-s)^{\alpha-1} 
		\| S_{t-s} [\mathcal{I}_\alpha(s)]\|_{L^\infty} \, \d s. 
	\end{equation} 
	By the Sobolev embedding theorem --- see, for example,
	Grafakos \cite[Theorem 6.2.4]{Grafakos}) ---
	if we choose $\delta \in (0\,, 1)$ and $k\geq 2$ such that  
	\begin{equation}\label{eq:Sobolev}
		\delta k >1,
	\end{equation}
	then 
	\[ 
		\| S_{t-s}[\mathcal{I}_\alpha(s)]\|_{L^\infty}
		\leq C(\delta) \| S_{t-s} [\mathcal{I}_\alpha(s)]\|_{H^{\delta, k}(\R)},
	\]
	where $H^{\delta, k}(\R)$ denotes the space of Bessel potentials
	and $C(\delta)>0$ is a number that depends only on $\delta$. 
	Since $H^{\delta, k}(\R)$ coincides with the complex interpolation space
	$[L^k, W^{1, k}(\R)]_{\delta}$ ---
	see, for example, Lunardi \cite[Example 2.12]{Lunardi})
	--- we may appeal to Corollary 2.8 of Lunardi ({\it ibid.}) in order to 
	be able to say that 
	\[ 
		 \| S_{t-s}[\mathcal{I}_\alpha(s)]\|_{H^{\delta, k}(\R)}
		 \leq \|  S_{t-s}[\mathcal{I}_\alpha(s)]\|_{L^k}^{1-\delta} 
		 \| S_{t-s}[\mathcal{I}_\alpha(s)]\|_{W^{1,k}(\R)}^\delta. 
	\]
	The heat semigroup is a contraction mapping on $L^k$. Therefore,
	\[ 
		\|  S_{t-s}[\mathcal{I}_\alpha(s)]\|_{L^k} \leq \| \mathcal{I}_\alpha(s)\|_{L^k},
	\]
	and a direct calculation shows that there exists some constant $C>0$ 
	--- independent of the parameters $(s\,, t\,, \alpha\,, k\,, \delta)$ --- such that 
	 \[ 
	 	\| S_{t-s} [\mathcal{I}_\alpha(s)]\|_{W^{1,k}(\R)} \leq C (t-s)^{-1/2} \,  
		\| \mathcal{I}_\alpha(s)\|_{L^k}.
	\]
	We can combine our efforts so far in order to obtain the following:
	\begin{equation} \label{eq:sup_bound} 
		\| S_{t-s}[\mathcal{I}_\alpha(s)]\|_{L^\infty}  \leq  C(\delta) 
		(t-s)^{-\delta /2}\, \| \mathcal{I}_\alpha(s)\|_{L^k}.
      \end{equation}
	Therefore, if we assume additionally that
	$\delta < 2\alpha$ and $k\left(\alpha-\frac\delta2\right)>1$,  
	then we deduce from \eqref{eq:bound1}  that 
	\begin{equation}\label{id:above}\begin{aligned}
		  \| \mathcal{I} (t)\|^k_{L^\infty} &\leq  
		  	\left| C(\delta) \int_0^t (t-s)^{\alpha-1-\frac{\delta}{2}} \,  
			\| \mathcal{I}_\alpha(s)\|_{L^k}\,\d s\right|^k\\
		  & \leq [C(\delta) ]^k
		  	\left| \int_0^t (t-s)^{k(\alpha-1-\frac{\delta}{2})/(k-1)} \, \d s \right|^{k-1}
		  	\, \int_0^t  \| \mathcal{I}_\alpha(s)\|_{L^k}^k \, \d s \\
		  &= [C(\delta) ]^k
		  	\left[ \frac{  (k-1)}{k\left(\alpha-\frac{\delta} {2}\right)-1}\right]^{k-1}\, 
		  	t^{k\left(\alpha - \frac{\delta}{2} \right) -1} \,  \int_0^t  \| \mathcal{I}_\alpha(s)\|_{L^k}^k \, \d s. 
	\end{aligned}\end{equation} 
	In the second inequality of \eqref{id:above} 
	we used H\"older's inequality. We now choose $\alpha=\alpha(\delta)$ judiciously 
	--- for instance, $\alpha=2\delta$ --- in order to deduce from the above
	that
	\[
		 \frac{(k-1)}{k\left(\alpha-\frac{\delta}{2}\right)-1} \leq C_1,
	\]
	for a constant $C_1=C_1(\delta)>0$. With our present choice in mind, we find that
	\begin{equation}\label{eq:sup_t}
		 \E \left[ \sup_{0\leq s\leq t}  \| \mathcal{I} (t)\|^k_{L^\infty} \right] \leq 
		[C(\delta) ]^k
		t^{k\left(\alpha - \frac{\delta}{2} \right) -1} \,  \int_0^t  
		\E\left[ \| \mathcal{I}_\alpha(s)\|_{L^k}^k\right] \, \d s.
	 \end{equation} 
	It remains to estimate $\E [\| \mathcal{I}_\alpha(s)\|_{L^k}^k ]$
	for $0<s<t$.
	 
	We can combine the Burkolder-Davis-Gundy  inequality, \eqref{eq:worthy}, and the Minkowskii inequality, 
	in order to see that there exists a number $C:=C(\lip_\sigma)>0$  such that 
	\begin{align*}
		\E \left[|\mathcal{I}_\alpha(t\,, x)|^k\right] &  \leq C^k k^{k/2} 
			\, \E \left| \int_0^t \d s\,  (t-s)^{-2\alpha} \int_{-\infty}^\infty\d y \, 
			\left[p_{t-s}(y-x) \, v(s\,, y)\right]^2   \right|^{k/2} \\
		& \leq C^k k^{k/2}  \left|  \int_0^t \d s \, (t-s)^{-2\alpha} \int_{-\infty}^\infty \d y\, 
			[p_{t-s}(y-x)]^2 \|  v(s\,, y)\|_k^2 \right|^{k/2}\\
		&\leq  C^k k^{k/2}  \left|  \int_0^t \d s \, (t-s)^{-2\alpha-\frac{1}{2}}
			\int_{-\infty}^\infty \d y\ p_{t-s}(y-x) \| v(s\,, y))\|_k^2 \right|^{k/2},
	\end{align*}
	where we used the elementary inequality $\sup_{x\in\R}p_r(x)\le r^{-1/2}$, valid for all $r>0$, in the last line.
	We now can use Lemma \ref{lem:moment} and the 
	contraction property of the heat semigroup on $L^k$  in order to see that if $\alpha<1/4$, then 
	\begin{align}\notag
		\left(\E\int_{-\infty}^\infty  |\mathcal{I}_\alpha(t\,, x)|^k\, \d x\right)^{1/k} &  \lesssim\sqrt k\,
		     \|v_0\|_{L^\infty}^{1/2} \, \exp\left( Ak^2 t\right)  \\\notag
		&\qquad \times\int_{-\infty}^\infty \, \d x  \left|  \int_0^t \frac{\d s}{% 
			(t-s)^{2\alpha+\frac{1}{2}}}\int_{-\infty}^\infty \d y\, p_{t-s}(y-x) 
			\left(S_s v_0\right)(y) \right|^{1/2} \\
		  &\lesssim \sqrt k\,\|v_0\|_{L^\infty}^{1/2} \, \exp\left( Ak^2 t\right)\, 
		  	\left| \int_0^t \frac{\d s}{% 
			(t-s)^{2\alpha+\frac{1}{2}}}\right|^{1/2}  \left\| S_t v_0  \right\|_{L^{k/2}}^{1/2} \\\notag
		  &\lesssim\sqrt k\,
		 	 \|v_0\|_{L^\infty}^{1/2} \, \exp\left( Ak^2 t\right)\, (1-4\alpha)^{-1/2} \,
		  	t^{\left(\frac{1}{2}-2\alpha\right)/2} \, \|v_0\|_{L^{k/2}}^{1/2}\\
		  &\lesssim\sqrt k\,
		  	\, \exp\left( Ak^2 t\right)\,(1-4\alpha)^{-1/2} \, 
			t^{\left(\frac{1}{2}-2\alpha\right)/2} \, \|v_0\|_{L^\infty}^{(k-1)/k} \|v_0\|_{L^1}^{1/k},
			\notag
	\end{align}
	where the implicit constants do not depend on $(t\,,k)$.
	Therefore, \eqref{eq:sup_t} yields real numbers $C_i:=C_i(\delta\,, \alpha)>0$ 
	[$i=1,2$] such that 
	\begin{align*}
		\E \left[ \sup_{0\leq s\leq t}  \| \mathcal{I} (t)\|^k_{L^\infty} \right]& 
			\leq C_1^k k^{k/2}  \, \exp\left( Ak^3 t\right)\,
			t^{k\left(\alpha - \frac{\delta}{2} \right) -1}\,  \|v_0\|_{L^\infty}^{k-1} 
			\|v_0\|_{L^1} \int_0^t s^{\left(\frac{1}{2}-2\alpha\right)k/2} \, \d s\\
		&\leq C_2^k k^{k/2}  \, \exp\left( Ak^3 t\right)\,  
			t^{k\left(\frac{1}{4} - \frac{\delta}{2} \right)} \,  \|v_0\|_{L^\infty}^{k-1} \|v_0\|_{L^1}.  
	\end{align*}
	We now choose $\theta:=\frac{1}{4} - \frac{\delta}{2}$ and $\alpha:=2\delta$ 
	in order to obtain \eqref{eq:moment_noise} for $\theta \in (3/16\,, 1/4)$.
	The lower bound $3/16$ comes from the assumption that $\alpha <1/4$. For smaller 
	values of $\theta$, we merely observe that
	\eqref{eq:moment_noise} holds automatically; this is because $t \in (0\,, 1]$,
	Because  the condition \eqref{eq:Sobolev} is equivalent to $k(1-4\theta)>2$,
	this completes the proof of Lemma \ref{lem:tp2}.
\end{proof}

\begin{lemma}\label{lem:pr2}
	For every $0<\theta<\frac14$ there exists
	$c=c(\theta\,,\lip_\sigma)>0$ such that
	\begin{align*}
		\left\| \|v(t)\|_{L^\infty} \right\|_k
			&\lesssim \sqrt{k}\,\exp( c k^2t) \|v_0\|_{L^\infty}^{1-1/k}\,  t^{\theta}
			+   t^{-1/2},\\
		\left\| \sup_{s\in(0,t)}\|v(s)\|_{L^\infty} \right\|_k
			&\lesssim \sqrt{k}\,\exp( c k^2t) \|v_0\|_{L^\infty}^{1-1/k} \,  t^\theta
			+ \|v_0\|_{L^\infty},
	\end{align*}
	uniformly for all $t\in(0\,,1]$,  all nonnegative functions $v_0\in L^1\cap L^\infty$ 
	that satisfies $\|v_0\|_{L^1}=1$, and all real numbers $k >2/(1-4\theta)$. 
\end{lemma}

\begin{proof}
	Since $\int_{-\infty}^\infty v_0(x)\,\d x=1$ and $\sup_{x\in \R} p_t(x) \leq t^{-1/2}$,
	it follows that $\|S_tv_0\|_{L^\infty}\lesssim \|v_0\|_{L^\infty} \wedge t^{-1/2}$.
	The first portion of the lemma follows from this observation,   Lemma \ref{lem:tp2}, and \eqref{mild}. The second portion follows similarly.
\end{proof}

\subsection{Control of tall peaks and total mass}
In this section we assume that $v$ denotes the unique solution to \eqref{eq:SHE2}, 
and show that the tall peaks and total mass of $v$  do not move much for a short time. 

\begin{proposition}\label{pr:peaks}
	For  $\frac43<\gamma<2$,  there exist $C=C(\gamma\,,\lip_\sigma)>0$ 
	such that
	\begin{align*}
		\P\left\{ \left\| v\left( N^{-\gamma} \right)
		\right\|_{L^\infty}  \ge  N \right\}
		&\le C\exp\left( - N^{(3\gamma-4)/2} \right),\\
		\P\left\{  \sup_{0\le s\le N^{-\gamma} }\|v(s)\|_{L^\infty} \ge 
		2N\right\} & \le C\exp\left( - \tfrac12 N^{(3\gamma-4)/2} \right),
	\end{align*}
	uniformly for all real numbers $N\ge1$, 
	  and all nonnegative functions $v_0$
	that satisfy $\|v_0\|_{L^1}=1$ and $\|v_0\|_{L^\infty}\le N$.
\end{proposition}

\begin{proof}
For each fixed  $\gamma \in (4/3\, , 2)$, we can  choose and fix $0<\theta<\frac14$ that satisfies $\gamma(3-4\theta) < 4$.
	We now   apply Lemma \ref{lem:pr2} with $t=N^{-\gamma}$ 
	and $k= N^{(3\gamma-4)/2}$ to see that there exists a
	real number $C_1=C_1(\gamma\,,\theta\,,\lip_\sigma)>0$ such that
	\begin{equation}\label{Joob}\begin{split}
		\left\| \|v(N^{-\gamma})\|_{L^\infty} \right\|_{N^{(3\gamma-4)/2}}
			&\le \tfrac12 C_1N^{\gamma(3-4\theta)/4} + N^{\gamma/2},\\
		\left\| \sup_{0\le s\le N^{-\gamma}}\|v(s)\|_{L^\infty} \right\|_{N^{(3\gamma-4)/2}}
			&\le \tfrac12 C_1N^{\gamma(3-4\theta)/4}+N,
	\end{split}\end{equation}
	uniformly for all large $N$,
	and all nonnegative functions $v_0\in L^1\cap L^\infty$ that satisfy
	$\|v_0\|_{L^1}=1$ and $\|v_0\|_{L^\infty}\leq N$. 	By the first bound in \eqref{Joob} and Chebyshev's inequality, we have 
	\begin{align*}
		\P\left\{ \left\| v\left( N^{-\gamma} \right)
			\right\|_{L^\infty}  \ge N\right\}
			&\le  \E\left( \left| \frac{\left\| v\left( N^{-\gamma} \right)
			\right\|_{L^\infty}}{N}\right|^k\right) \\
		&\le \exp\left( - N^{(3\gamma-4)/2} \right)\qquad\text{for
			all $N\ge \e \left(\tfrac12 C_1N^{\gamma(3-4\theta)/4} + N^{\gamma/2} \right) $}.
	\end{align*}	
	Note that since $\gamma(3-4\theta) < 4$ and $\gamma<2$, there exists $N_0>0$ such that $N\ge \e \left(\tfrac12 C_1N^{\gamma(3-4\theta)/4} + N^{\gamma/2} \right) $ for all $N\geq N_0$. Thus, we find that there exists $C_2=C_2(\gamma\,,\theta\,,\lip_\sigma)>0$ such that
\[
	\P\left\{ \left\| v\left( N^{-\gamma} \right)
	\right\|_{L^\infty}  \ge K_1 N^{\gamma/2} \right\}
	\le K_2\exp\left( - N^{(3\gamma-4)/2} \right)\qquad\text{for
	all $N\ge 1$},
\]
which results in  the first assertion of Proposition \ref{pr:peaks}.
	
For  the second portion of Proposition \ref{pr:peaks}, since 
$\gamma(3-4\theta)/4 \in (0,1 )$, it follows from \eqref{Joob} that for 
every $q>1$ there exists $N_0=N_0(q\,,\gamma\,,\theta\,,\lip_\sigma)>0$ such that
\[
	\left\| \sup_{0\le s\le N^{-\gamma}}\|v(s)\|_{L^\infty} \right\|_{N^{(3\gamma-4)/2}}
	\le qN\qquad\text{for all $N\ge N_0$}.
\]
We  now choose $q=2\exp(-1/2)$, and then use Chebyshev's inequality to get that
\[
	\P\left\{ \sup_{0\le s\le N^{-\gamma}}\|v(s)\|_{L^\infty} 
	\ge 2N \right\}\le \exp\left( -\tfrac12 N^{(3\gamma-4)/2}\right)\qquad\text{for all $N\ge N_0$}.
\]
This implies that  there exists $C_3=C_3(\gamma\,,\theta\,,\lip_\sigma)>0$ such that

\[
	\P\left\{ \sup_{0\le s\le N^{-\gamma}}\|v(s)\|_{L^\infty} 
	\ge 2N \right\}\le 
	C_3\exp\left( -\tfrac12 N^{(3\gamma-4)/2}\right)\qquad\text{for all $N\ge1$},
\]
which completes the proof of Proposition \ref{pr:peaks}.
\end{proof}

\begin{proposition}\label{pr:valleys}
	For every $\frac43<\gamma<2$ there exists $L=L(\gamma\,,\lip_\sigma)>1$ such that
	\[
		\P\left\{ \inf_{0\le t\le N^{-\gamma}} \|v(t)\|_{L^1} 
		\le  \frac{1}{2} \quad \text{or}\   \sup_{0\le t\le N^{-\gamma}}
		\|v(t)\|_{L^1} \ge  2 \right\}
		\le L\exp\left( -\frac{N^{(3\gamma-4)/3}}{L} \right),
	\]
	uniformly for all real numbers $N\ge1$ and all nonnegative functions $v_0\in L^1\cap L^\infty$
	that satisfy $\|v_0\|_{L^1}=1$ and $\|v_0\|_{L^\infty}\le N$.
\end{proposition}

\begin{proof}
	Since $\tilde\xi$ can be regarded as a worthy martingale measure, 
	we can conclude from\eqref{eq:worthy} that $\|v(t)\|_{L^1}$ is a continuous 
	$L^2(\Omega)$-martingale whose quadratic variation at time $t>0$ is
	bounded by $\lip_\sigma^2 \int_0^t \|v(s)\|_{L^2(\R)}^2 \d s$. 
	We now simply follow the proof of  Proposition 3.7 of \cite{KKM}.
	The only difference is that the present Lemma \ref{lem:moment} can be used  
	for the estimates of  the moments of $v$; one can easily see that this change 
	does not affect the validity of our claim and deduce the proposition.
	The details of the proof are left to the interested reader.  
\end{proof}

\subsection{Proof of Proposition \ref{prop:WLOG} }
We now prove Proposition \ref{prop:WLOG}. Suppose  $v$ denotes the unique solution to 
\eqref{eq:SHE2} subject to initial data $v_0 \in C^+_{\textit{rap}}(\R)$. 
Since $v_0 \in C^+_{\textit{rap}}(\R)$, it follows that 
$v(t) \in C^+_{\textit{rap}}(\R)$ for all $t\geq 0$ almost surely.
For details, see for example Shiga \cite[Theorem 2.5]{Shiga}.\footnote{%
	In Shiga's work  \cite[Theorem 2.5]{Shiga}, the noise part is $\tilde\sigma(t\,, x\,, u\,, \omega) 
	\, \xi(t\,, x)$ where $\xi$ denotes space-time white noise and $\tilde\sigma$ is a random 
	function  that satisfies certain regularity conditions. 
	Here,  we  can regard    $v(t\,,x) \tilde\xi(t\,,x) := v(t\,, x)\frac{\sigma(u(t\,, x))}{u(t\,, x)} \, \xi(t\,, x)$ 
	as $ \tilde\sigma(t, x, v, \omega) \xi(t, x)$
	and immediately conclude that  $\tilde\sigma$  satisfies the certain regularity 
	conditions of Shiga \cite{Shiga}.}. 
Thus, we can see that $v(t) \in L^1\cap L^\infty$ for all $t\geq 0$ almost surely. 
We are ready to proceed with the proof.

\begin{proof}[Proof of Proposition \ref{prop:WLOG}]
	This proof is similar to the proof of \cite[Theorem 1.3]{KKM}.
	
	We first   define a stopping time $\tau(n)$ for every $n\geq 1$ as 
	 \[
		\tau(n) = \inf\left\{ t>0:\, \| v(t)\|_{L^\infty} \ge  
		 n^{\beta/2} \| v(t)\|_{L^1}\right\}
		\qquad[\inf\varnothing=\infty].
	\] 
	Then, since $\P(A\cap B) \leq \P(A\mid B)$, 
	\begin{align*}
		\P\left\{ \sup_{0\le t\le n}\frac{\|v(t)\|_{L^\infty}}{
			\|v(t)\|_{L^1}} \ge n^{\beta}\right\} 
			&=
			\P\left\{ \tau(n)<n\ , \sup_{\tau(n) \le t\le n}\frac{\|v(t)\|_{L^\infty}}{
			\|v(t)\|_{L^1}} \ge n^{\beta} \right\}\\
		&\le \P\left(  \left. \sup_{\tau(n) \le t\le n}\frac{\|v(t)\|_{L^\infty}}{
			\|v(t)\|_{L^1}}  \ge n^{\beta}\   \right| \,\tau(n)<\infty\right). 
	\end{align*}
	Since $v(t) \in C^+_{\textit{rap}}(\R)$ for all $t\geq 0$  almost surely, and 
	because $t\mapsto v(t)$ is continuous in time, 
	\[
		\|v(\tau(n))\|_{L^\infty} =   n^{\beta/2} \|v(\tau(n))\|_{L^1}
		\qquad\text{a.s.\ on $\{\tau(n)<\infty\}$.}
	\]
	Thus, by the strong Markov property of $(u\,,v)$ [see Lemma \ref{lem:SMP}],
	\[ 
	\P\left(  \left. \sup_{\tau(n) \le t\le n}\frac{\|v(t)\|_{L^\infty}}{
		\|v(t)\|_{L^1}}  \ge n^{\beta}\,   \right| \, \mathcal{F}_{\tau(n)}\right)  \leq 
		\sup_{\substack{\tilde v_0\in C^+_b(\R):\\
		\|\tilde v_0\|_{L^\infty} =  n^{\beta/2}  \|\tilde v_0\|_{L^1}}}
		  \P\left\{ \sup_{0\le t\le n}\frac{\|\tilde v(t)\|_{L^\infty}}{
			\|\tilde v(t)\|_{L^1}} \ge n^{\beta}\right\} 
	\] 
	where $\tilde v$ denotes the solution to \eqref{eq:tilde v} 
	subject to initial data $\tilde v_0$ that is being optimized under 
	``$\sup_{\tilde{v}_0\in C^+_b(\R)}$\dots .''  Suppress the dependence of
	the following on the parameter $n$, and define
	\[
		\tilde V(t\,,x) = \frac{\tilde v(t\,,x)}{\|\tilde v_0\|_{L^1}}
		\qquad\text{for all $t\ge0$ and $x\in\R$}.
	\]
	The random field $\tilde V=\{\tilde V(t\,,x)\}_{t\ge0,x\in\R}$ solves the SPDE,
	\begin{equation}\label{V}
		\partial_t \tilde V(t\,,x) =\tfrac12 \partial^2_x \tilde V(t\,,x) + 
		\tilde{V}(t\,,x)\hat\xi(t\,, x)\,  
		\qquad\text{on $(0\,,\infty)\times\R$},
	\end{equation}
	subject to $\tilde V(0\,,x)= \tilde V_0(x)$, where 
	$\hat\xi(t\,,x)=\theta_{\tau(n)}\tilde{\xi}(t\,,x)$,  
	$\|\tilde V_0\|_{L^\infty} = n^{\beta/2}$, and  $\|\tilde V_0\|_{L^1}=1.$
	
	Note that $\hat \xi$ can be considered as a worthy martingale measure 
	whose dominating measure is bounded by constant multiple of the
	Lebesgue measure  (as in \eqref{eq:worthy}). 
	Thus,  we can use all the results from  the previous sections. 
	
	Define 
	\[
		N:=n^{\beta/2}.
	\]
	Then,
	\begin{align*}
		\P\left\{ \sup_{0\le t\le n}\frac{\|v(t)\|_{L^\infty}}{
			\|v(t)\|_{L^1}} \ge n^{\beta}\right\} &  \leq 
			\sup_{\substack{\tilde V_0\in C^+_b(\R):\\
			\|\tilde V_0\|_{L^\infty} = n^{\beta/2}, \,   \|\tilde V_0\|_{L^1}=1}}
			\P\left\{ \sup_{0\le t\le n }\frac{\|\tilde V(t)\|_{L^\infty}}{
			\|\tilde V(t)\|_{L^1}} \ge n^\beta  \right\}\\
		&=\sup_{\substack{\tilde V_0\in C^+_b(\R):\\
			\|\tilde V_0\| _{L^\infty} = N, \,   \|\tilde V_0\|_{L^1}=1}}
			\P\left\{ \sup_{0\le t\le N^{2/\beta} }\frac{\|\tilde V(t)\|_{L^\infty}}{
			\|\tilde V(t)\|_{L^1}} \ge N^2\right\} \\
		&\leq A_1 + A_2,
	\end{align*}
	where 
	\begin{align*}
		A_1&:= \sup_{\substack{\tilde V_0\in C^+_b(\R):\\
			\|\tilde V_0\| _{L^\infty} = N, \,   \|\tilde V_0\|_{L^1}=1}}
			\P\left\{ \sup_{0\le t\le N^{-\gamma}  }\frac{\|\tilde V(t)\|_{L^\infty}}{
			\|\tilde V(t)\|_{L^1}} \ge N^2\right\},\\
		A_2&:= \sup_{\substack{\tilde V_0\in C^+_b(\R):\\
			\|\tilde V_0\| _{L^\infty} = N, \,   \|\tilde V_0\|_{L^1}=1}}
			\P\left\{ \sup_{N^{-\gamma} \le t\le N^{2/\beta} }\frac{\|\tilde V(t)\|_{L^\infty}}{
			\|\tilde V(t)\|_{L^1}} \ge N^2\right\}. 
	\end{align*}
	We first bound $A_1$ as follows: 
	\begin{align*} 
		A_1 &\leq \sup_{\substack{\tilde V_0\in C^+_b(\R):\\
			\|\tilde V_0\| _{L^\infty} = N, \,   \|\tilde V_0\|_{L^1}=1}} \P\left\{ \sup_{0\le t\le N^{-\gamma}} 
			\|\tilde V(t)\|_{L^\infty(\T)} \ge 2N\right\}\\
		&\ + \sup_{\substack{\tilde V_0\in C^+_b(\R):\\
			\|\tilde V_0\| _{L^\infty} = N, \,   \|\tilde V_0\|_{L^1}=1}}\P\left\{ 
			\sup_{0\le t\le N^{-\gamma}} \| \tilde V(t)\|_{L^\infty(\T)} < 2N \ ,
			\inf_{0\le t\le N^{-\gamma}} \|\tilde V(t)\|_{L^1(\T)} \le \frac2N\right\}\\
		& \le\sup_{\substack{\tilde V_0\in C^+_b(\R):\\
			\|\tilde V_0\| _{L^\infty} = N, \,   \|\tilde V_0\|_{L^1}=1}}
			\left[ \P\left\{ \sup_{0\le t\le N^{-\gamma}} 
			\|\tilde V(t)\|_{L^\infty(\T)} \ge 2N\right\}
			+ 
			\P\left\{\inf_{0\le t\le N^{-\gamma}} \|\tilde V(t)\|_{L^1(\T)} \le \frac2N\right\} \right] \\
		&\le C_1 \exp\left( - C_1^{-1} N^{(3\gamma -4)/3} \right) \quad \text{for all $N\geq 1$}.
	\end{align*}
	In the last inequality, we used Propositions \ref{pr:peaks} and \ref{pr:valleys}, 
	and the constant $C_1$ only depends on $\gamma$ and $\lip_\sigma$.  
	
	Regarding  $A_2$, we have 
	\begin{align*}
		A_2 &\leq \sup_{\substack{\tilde V_0\in C^+_b(\R):\\
			\|\tilde V_0\| _{L^\infty} = N, \,   \|\tilde V_0\|_{L^1}=1}}
			\P\left\{ \|\tilde V(N^{-\gamma})\|_{L^\infty(\T)} \ge N\right\}\\
		&\hskip.5in +\sup_{\substack{\tilde V_0\in C^+_b(\R):\\
			\|\tilde V_0\| _{L^\infty} = N, \,   \|\tilde V_0\|_{L^1}=1}}
			\P\left\{\|\tilde V(N^{-\gamma})\|_{L^\infty(\T)} \le N\quad 
			\text{and}\quad  \sup_{N^{-\gamma} \le t\le N^{2/\beta} }\frac{\|\tilde V(t)\|_{L^\infty}}{
			\|\tilde V(t)\|_{L^1}} \ge N^2\right\} \\
		&\le K_2\exp\left( - N^{(3\gamma-4)/2} \right) +\hskip-5mm 
			\sup_{\substack{\tilde V_0\in C^+_b(\R):\\
			\|\tilde V_0\| _{L^\infty} = N, \,   \|\tilde V_0\|_{L^1}=1}}  
			\P\left\{ \sup_{0 \le t\le N^{2/\beta} -N^{-	\gamma} }\frac{\|\tilde V(t)\|_{L^\infty}}{
			\|\tilde V(t)\|_{L^1}} \ge N^2\right\} \quad \text{for all $N\geq 1$}. 
	\end{align*} 
	We used Proposition \ref{pr:peaks} in the last inequality,
	and we note that the constant $K_2$ depends only on $\gamma$ and $\lip_\sigma$. 
	In addition, we also conditioned on $\mathcal{F}_{\tau\left(N^{2/\beta}\right)+N^{-\gamma}}$
	and consider  a new SPDE of the form \eqref{V} using the SMP [Lemma \ref{lem:SMP}]. 
	 
	Combine the preceding bounds for $A_1$ and $A_2$ and repeat the above  process  
	to obtain a real number $C=C(\gamma\,, \lip_\sigma)>1$ such that
	\begin{align*}
		& \sup_{\substack{\tilde V_0\in C^+_b(\R):\\
			\|\tilde V_0\| _{L^\infty} = N, \,   \|\tilde V_0\|_{L^1}=1}}
			\P\left\{ \sup_{0\le t\le N^{2/\beta} }\frac{\|\tilde V(t)\|_{L^\infty}}{
			\|\tilde V(t)\|_{L^1}} \ge N^2\right\} \\
		&\leq  C \exp\left( - C^{-1}  N^{(3\gamma-4)/3} \right)	+ \sup_{\substack{\tilde V_0\in C^+_b(\R):\\
			\|\tilde V_0\| _{L^\infty} = N, \,   \|\tilde V_0\|_{L^1}=1}}
			\P\left\{ \sup_{0 \le t\le N^{2/\beta} -N^{-\gamma} }\frac{\|\tilde V(t)\|_{L^\infty}}{
			\|\tilde V(t)\|_{L^1}} \ge N^2\right\}\\
		&\leq 2C \exp\left( - C^{-1}  N^{(3\gamma-4)/3} \right)	+ \sup_{\substack{\tilde V_0\in C^+_b(\R):\\
			\|\tilde V_0\| _{L^\infty} = N, \,   \|\tilde V_0\|_{L^1}=1}}
			\P\left\{ \sup_{0 \le t\le N^{2/\beta} -2N^{-\gamma} }\frac{\|\tilde V(t)\|_{L^\infty}}{
			\|\tilde V(t)\|_{L^1}} \ge N^2\right\} \\
		&  \qquad \qquad \vdots\\		 
		&\leq\ell_N C \exp\left( - C^{-1}  N^{(3\gamma-4)/3} \right),
	\end{align*} 
	where $\ell_N = \lfloor  N^{\gamma+2/\beta} \rfloor + 1$. Since $N=n^{\beta/2}$, 
	\[
		\P\left\{ \sup_{0\le t\le n}\frac{\|v(t)\|_{L^\infty}}{
		\|v(t)\|_{L^1}} \ge n^{\beta}\right\}  \leq C\left(1+ n^{1+\beta\gamma/2}  \right)
		\exp\left(  -C^{-1} n^{(3\gamma-4)\beta/6} \right). 
	\]
	Our choices of $\gamma$ and $\beta \geq 6/(3\gamma -4)$
	yield  the proof of Proposition \ref{prop:WLOG}. 
\end{proof}

\section{Proof of Theorem \ref{th:SHE-cc}: Uniform dissipation} \label{sec:dissipation}

In this section, we show the solution $v$ to \eqref{eq:SHE2}  dissipates uniformly in $x$ when $v_0(x)$ decays rapidly as $|x|\to\infty$ (see Theorem \ref{th:SHE-v}).  This also provides the proof of  Theorem \ref{th:SHE-cc}.  

\begin{theorem} \label{th:SHE-v}
	Suppose  $v$ denotes the unique solution to  \eqref{eq:SHE2} with  initial data
	$v_0 \in C_b^+(\R)$ that satisfies $\limsup_{|x|\to\infty} x^{-2} \log v_0(x) <0$ and $\|v_0\|_{L^1}>0$. 
	Then, there exists  a constant $\Lambda_3>0$  that only depends on ${\rm L}_\sigma$ and $\lip_\sigma$,  
	and an almost surely finite random number $T>0$ such that
	\[
		\sup_{x\in \R}v(t\,, x) \le \exp\left(-\Lambda_3t^{1/3}\right)
		\qquad\text{for all $t>T$}.
	\]
\end{theorem}

Before we proceed with a proof, let us make two quick remarks.

\begin{remark}\begin{compactenum}[(a)]
	\item Every initial function $v_0 \in C_b^+(\R)$ that satisfies 
		$\limsup_{|x|\to\infty} x^{-2} \log v_0(x) <0$ is automatically an element of
		$C^+_{\textit{rap}}(\R)$.  
	\item Theorem \ref{th:SHE-v} is about  uniform dissipation. We establish
		dissipation by first investigating the long-time behavior of  the total mass process 
		$t\mapsto\|v(t)\|_{L^1}$. 
\end{compactenum}\end{remark}

The proof of Theorem \ref{th:SHE-v} hinges on the following quantitative probability estimate.

\begin{proposition}\label{prop:L1}
	There exist constants $\Lambda_4>0$ and $\Lambda_5>0$
	that only depend on $ {\rm L}_\sigma$ and  $\lip_\sigma$ and satisfy
	\begin{equation}\label{eq:L1bound}
		\P \left\{  \sup_{s\geq t} \|v(s)\|_{L^1} \geq 
		\exp\left(- \Lambda_4 t^{1/3}\right)  \right\} \leq \exp\left(-\Lambda_5 t^{1/3}\right)
		\quad\text{for every $t>0$}.
	\end{equation} 
\end{proposition}

\begin{proof}
	Recall that, off a single null set,  
	$v(t\,,x)\ge0$ for all $t\geq 0$ and $x\in \R$.
	In fact, because of condition \eqref{eq:lip-sigma-condition},  the  argument of \cite{Mueller1} 
	can be used virtually unchanged in order to prove that, off a single null set,
	$v(t\,,x)>0$ simultaneously for every $t>0$ and $x\in \R$. 
	Thus, just as we did in the proof of Proposition \ref{pr:valleys}, we
	may conclude that $t\mapsto \|v(t)\|_{L^1}$ is a strictly positive continuous $L^2(\Omega)$-martingale 
	whose  quadratic variation at every time $t>0$ is
	bounded below and above by constant multiples of $\int_0^t \|v(s)\|_{L^2}^2 \, \d s$ 
	that do not depend on $t$; see also \eqref{eq:worthy}. 
	We now follow the proof of Theorem 4.1 of Chen, Cranston, Khoshnevisan, and Kim
	\cite[(4.20)]{CCKK}, essentially line by line,
	in order to learn that there exists a number $C=C({\rm L}_\sigma\,, \lip_\sigma)>0$ such that 
	\[
		\E \sqrt{ \|v(t)\|_{L^1}} \leq C \exp\left( - C t^{1/3}\right)
		\qquad\text{for every $t>0$}.
	\]
	Proposition \ref{prop:L1} follows from this and Doob's maximal inequality for supermartingales.
\end{proof}

We now use Proposition \ref{prop:L1}, along with Proposition \ref{prop:WLOG},
in order to prove Theorem \ref{th:SHE-v}.

\begin{proof}[Proof of Theorem \ref{th:SHE-v}]
	Let $\Lambda_4>0$ and $\Lambda_5>0$ denote the constants 
	that were given in Proposition \ref{prop:L1}, and let $\alpha>0$ be a real number that will be chosen later.  
	Recall Proposition  \ref{prop:WLOG} and  choose $\gamma=5/3$ and $\beta=6$ in 
	order to find that 
	\begin{equation}
		\P \left\{   \sup_{n-1 \leq t \leq n} \| v(t)\|_{L^\infty}  
		\geq \exp\left(-\alpha n^{1/3} \right)   \right\}  \leq  A_1+A_2,
	\end{equation}
	where
	\begin{align}
		A_1&:= \P\left\{ \sup_{0\le t\le n}\frac{\|v(t)\|_{L^\infty}}{\|v(t)\|_{L^1}} \ge n^6 \right\}, \\
		A_2	&:= \P \left\{  \sup_{n-1 \leq t \leq n } \| v(t)\|_{L^\infty}  
			\geq \exp\left(-\alpha n^{1/3} \right)   \quad \text{and}\quad  
			\sup_{0 \le t\le n}\frac{\|v(t)\|_{L^\infty}}{\|v(t)\|_{L^1}} \le n^6 \right\}.
	\end{align}
	According to Proposition \ref{prop:WLOG}, there exists a constant $c=c(\lip_\sigma)>0$ such that 
	\begin{equation}\label{eq:A1_bound}
		A_1 \leq \exp(-c n) \quad \text{for all every $n\geq 1$}.
	\end{equation} 
	As regards  $A_2$,   let us choose $\alpha < \Lambda_4$ so that for all large $n\geq 2$, 
	\[ 
		n^{-6} \,\exp\left(-\alpha n^{1/3} \right)  \geq \exp\left(-\Lambda_4 (n-1)^{1/3} \right).
	\]  
	In this way, we can find that 
	\begin{equation}\label{eq:A2_bound} \begin{aligned}
		A_2& \leq \P\left\{  \sup_{n-1 \leq t \leq n} \| v(t)\|_{L^1}  
			\geq n^{-6} \,\exp\left(-\alpha n^{1/3} \right)     \right\}   \\
		&\leq \P\left\{  \sup_{ t \geq n-1 } \| v(t)\|_{L^1}  \geq  
			\exp\left(-\Lambda_4 (n-1)^{1/3} \right)     \right\} 
			\leq \exp\left(-\Lambda_5 (n-1)^{1/3}\right),
	\end{aligned}\end{equation}
	thanks to Proposition \ref{prop:L1}.  Therefore,  \eqref{eq:A1_bound} and \eqref{eq:A2_bound}
	together imply the existence of a number 
	$\beta=\beta({\rm L}_\sigma, \lip_\sigma)>0$ such that  for all large $n\geq 1$ 
	\begin{equation}\label{eq:sup_bound-1}
		 \P \left\{   \sup_{n-1 \leq t \leq n} \| v(t)\|_{L^\infty}  \geq 
		 \exp\left(-\alpha n^{1/3} \right)   \right\}  \leq \exp\left(-\beta n^{1/3}\right).  
	\end{equation}  
	This and an the Borel-Cantelli lemma yield Theorem 
	\ref{th:SHE-v}.  
\end{proof}

\section{Proof of Theorem \ref{th:SHE}} \label{sec:main-proof}
Finally, in this section we use the results from \S\S\ref{sec:partition}--\ref{sec:dissipation}
in order to prove Theorem \ref{th:SHE}. 

\begin{proof}[Proof of Theorem \ref{th:SHE}]
	Let $\eta_1,\eta_2>0$ be two real numbers whose numerical values 
	will be determined later on in \eqref{eq:eta}.  
	
	Let us first fix an integer $n\geq1$ and define 
	\[ 
		L(t):=\exp\left( \eta_1 t^{1/3}\right)  \quad \text{and} \quad 
		M:= M_n:=2 \lfloor L(n) \rfloor .  
	\] 
	As in Remark \ref{rem:partition}, we may write $u(t\,, x)$ --- with the initial 
	function $u_0\equiv 1$ --- as  
	\[ 
		u(t\,, x) = \sum_{i=-M}^{M-1} v^{(i)}(t\,, x) +  v^{(M)}(t\,, x),
	\]
	where $v^{(i)}$ and $v^{(M)}$ satisfy  \eqref{eq:SHE2} with respective initial data  
	$v^{(i)}_0$ and $v^{(M)}_0$  that are  continuous and non-negative functions such that 
	the support of $v^{(i)}_0$ is in $[i-1\,, i+1]$ for every $i=-M,\ldots,M-1$
	and the support of 
	$v^{(M)}_0$ is in $\R\setminus(-M\,, M)$. 
	
	For every $\rho \in (0\,,1)$,
	\[
	 	\P\left\{ \sup_{n-1\leq t\leq n} \, \sup_{|x| \leq L(t)} \, 
		\frac{u(t\,, x)}{\e^{-\eta_2 t^{1/3}}} \geq \rho \right\}  \leq A^{(1)}_n + A^{(2)}_n,
	\]
	where 
	\begin{align*}
		A^{(1)}_n&:= \sum_{i=-M}^{M-1} \P\left\{ \adjustlimits\sup_{n-1 \leq t\leq n}
			\sup_{x\in\R}  \, \frac{v^{(i)}(t, x)}{\e^{-\eta_2t^{1/3}}}  \geq \frac{\rho}{2(M+1)} \right\} \\
		A^{(2)}_n&:= \P\left\{\adjustlimits\sup_{n-1 \leq t\leq n}\sup_{|x|\leq L(t)}  \, 
			\frac{v^{(M)}(t, x)}{\e^{-\eta_2t^{1/3}}}  \geq \frac{\rho}{2(M+1)} \right\}. 
	\end{align*} 
	We first consider $A^{(1)}_n$. Since each $v^{(i)}(0)$ has a compact support, 
	we may use  \eqref{eq:sup_bound-1}. Let $\alpha>0$ and $\beta>0$ denote the constants 
	given in \eqref{eq:sup_bound-1}. We now choose and fix $\eta_1,\eta_2>0$ so that 
	\begin{equation}\label{eq:eta}
	 	\eta_1 < \beta \qquad \text{and} \qquad \eta_1+\eta_2 \leq \frac{\alpha}{2}.
	 \end{equation}
	Any such choice of $\eta_1$ and $\eta_2$ ensures that for all large $n\geq 1$,
	\begin{equation}\label{eq:An1} \begin{aligned}
		A^{(1)}_n & \leq \sum_{i=-M}^{M-1} \P\left\{  \sup_{n-1 \leq t\leq n}
			\left\| v^{(i)}(t)\right\|_{L^\infty}  \geq \frac{\rho}{8} 
			\exp\left[ - (\eta_1+\eta_2) \, n^{1/3}\right]  \right\}\\
		& \leq \sum_{i=-M}^{M-1} \P\left\{  \sup_{n-1 \leq t\leq n}
			\left\| v^{(i)}(t)\right\|_{L^\infty}  \geq \exp\left[ - 2(\eta_1+\eta_2) 
			\, n^{1/3}\right]  \right\}
			\leq 4 \exp\left[ - (\beta -\eta_1) n^{1/3}\right].
	\end{aligned}\end{equation}  
	
	We now consider $A_n^{(2)}$. First, let us recall  that
	$v^{(M)}_0(x)=0$ when $|x|\leq M \leq 2L(n) $ and $v^{(M)}_0(x) \in [0\,, 1]$
	in general. Consequently, for every $t\in [n-1\,, n]$,
	\begin{equation}\label{eq:vM} \begin{aligned} 
		 \sup_{|x| \leq L(n)}  \left( S_t v^{(M)}_0\right) (x) 
		 	&=\sup_{|x| \leq L(n)} \int_{-\infty}^\infty p_{t}(x-y)\,  v_0^{(M)}(y)\, \d y
			\leq   \sup_{|x| \leq L(n)} \int_{|y|\geq 2L(n) } p_{t}(x-y) \, \d y \\
		  &\leq 2 \exp\left( - \frac{L(n)}{4n}  \right) = 2 \exp\left( - \frac{\e^{\eta_1 n^{1/3}} } {4n}   \right), 
	  \end{aligned}\end{equation}
	  for every $n\in\Z_+$.
	 Next, we estimate 
	 \[
	 	\E \left[ \sup_{n-1 \leq t \leq n} \, \sup_{|x|\leq L(t)} \left|v^{(M)}(t\,, x) \right|^k \right]
	\]
	as follows:
	 \begin{align*}
		  \E \left[ \sup_{n-1 \leq t \leq n} \, \sup_{|x|\leq L(t)} \left|v^{(M)}(t\,, x) \right|^k \right] 
		  	& \leq \E \left[ \sup_{n-1 \leq t \leq n} \, \sup_{|x|\leq L(n)} 
			\left|v^{(M)}(t\,, x) \right|^k \right]  \\
		  & \leq  \sum_{i=-L(n)}^{L(n)-1} \E \left[ \sup_{n-1 \leq t \leq n} \, 
		  	\sup_{i \leq x \leq  i+1} \left|v^{(M)}(t\,, x) \right|^k \right] 
			\leq 2^{k/2} \left( B_n^{(1)} + B_n^{(2)}\right), 
	   \end{align*}
	 where 
	 \[ 
	 	B_n^{(1)}:=  \sum_{i=-L(n)}^{L(n)-1}  \E \left(\left|v^{(M)}(n\,, i) \right|^k \right)
		\quad \text{and}\quad 
		B_n^{(2)}:=  \sum_{i=-L(n)}^{L(n)-1} \E \left[ 
		\sup_{\substack{%
			n-1 \leq s_1,  s_2 \leq n  \\  i \leq x_1,  x_2  \leq i+1 \\ (s_1, x_1)\neq (s_2, x_2) }} \,  
		|v^{(M)}(s_1, x_1) -v^{(M)}(s_2\,, x_2)|^k  \right].    
	 \]
	 In order to estimate $B_n^{(1)}$ we use Lemma \ref{lem:moment} and \eqref{eq:vM} as follows:
	 \begin{equation}\label{eq:B1n}\begin{aligned} 
		 B_n^{(1)} &\leq A \left\|v^{(M)}_0\right\|_{L^\infty}^{k/2} \,
		 	\exp(Ak^3n) \, \sum_{i=-L(n)}^{L(n)-1} \left( S_n v^{(M)}_0\right)^{k/2}  (i)  \\
		 &\leq 4A  \exp\left( \eta_1 n^{1/3}+  Ak^3n  - \frac{k\e^{\eta_1 n^{1/3}} } {8n}   \right). 
	 \end{aligned}\end{equation}
	Next we estimate  $B_n^{(2)}$.  
	
	Since $v^{(M)}_0 \in C^+_b(\R)$ with $\| v^{(M)}_0 \|_{L^\infty}=1$,  
	we use the well-known fact (see for example Khoshnevisan \cite[Chapter 5]{cbms}) 
	that  there exists a constant $L>0$ that is independent of $n$ and satisfies for all $k\geq 2$,
	\begin{equation}
		\sup_{\substack{n-1 \leq s_1 \neq  s_2 \leq n  \\  -\infty < x_1 \neq x_2 < \infty }}
		\E  \left(\left|  \frac{ v^{(M)}(s_1, x_1) - v^{(M)}(s_2\,, x_2)}{|s_1-s_2|^{1/4} + |x_1-x_2|^{1/2}}
		\right|^k\right) \leq L \e^{Lk^3 n} \left\| v^{(M)}_0 \right\|_{L^\infty}^k =L \e^{Lk^3 n}.
	\end{equation}
	On the other hand, Lemma \ref{lem:moment} and \eqref{eq:vM} together imply that, for all $k\geq 2$,
	\begin{equation}\begin{aligned}  
		&\sup_{\substack{n-1 \leq s_1,  s_2 \leq n  \\  -L(n) \leq  x_1,  x_2 \leq  L(n) }}
			\E \left(\left| v^{M)}(s_1, x_1) - v^{(M)}(s_2, x_2)  \right|^k \right) \leq 2^{k}
			\sup_{\substack{n-1 \leq s \leq n  \\  -L(n) \leq  x  \leq  L(n) }}
			\E \left( |v^{(M)}(s\,, x)|^k \right) \\
		&\hskip2in \leq (2A)^k  \left\| v^{(M)}_0 \right\|_{L^\infty}^{k/2}
			\exp\left( Ak^3 n\right) \sup_{|x|\leq L(n)} \left| \left( S_s v_0^{(M)} \right) (x)  \right|^{k/2} \\
		&\hskip2in\leq  2 (2A)^k \, \exp\left( Ak^3n - \frac{k\e^{\eta_1 n^{1/3}} } {8n}   \right). 
	\end{aligned}\end{equation}
	Since $\min\{a\,, b\} \leq (ab)^{1/2}$ for every $a>0, b>0$, there exists 
	a number $C_1>0$ that is independent of $n$ and satisfies the following
	for all $n-1\leq s_1\neq  s_2\leq n$ and all $i\leq x_1 \neq x_2 \leq i+1$  
	with $i\in [-L(n)\,, L(n)-1]$:
	\[  
		\E \left(\left| v^{(M)}(s_1, x_1) - v^{(M)}(s_2, x_2)  \right|^k\right) \leq C^k_1
		\exp\left( C_1k^3n - \frac{k\e^{\eta_1 n^{1/3}} } {16n}   \right)
		\left( |s_1 -s_2|^{1/8} + |x_1-x_2|^{1/4}\right)^{k}.  
	\]
	 A suitable form of the Kolmogorov continuity theorem \cite[Theorem C.6]{cbms}
	 implies that there exists a number $C_2>0$ -- independent of $n$ -- such that 
	 for all $k\geq 2$ and all $i\in [-L(n)\,, L(n)-1]$,
	\begin{equation}\label{eq:B2n} 
		\E \left[ \sup_{\substack{%
			n-1 \leq s_1,  s_2 \leq n  \\  i \leq x_1,  x_2  \leq i+1 \\ 
			(s_1, x_1)\neq (s_2, x_2) }} 
		\, \left|  \frac{ v^{(M)}(s_1, x_1) - v^{(M)}(s_2, x_2)}{|s_1-s_2|^{1/16} + |x_1-x_2|^{1/8}} \right|^k  \right] 
		\leq C^k_2  \exp\left( C_2k^3n - \frac{k\e^{\eta_1 n^{1/3}} } {16n}   \right).
	\end{equation} 
	By \eqref{eq:B1n} and \eqref{eq:B2n} (here we fix $k=2$) we get  that for all large $n\geq 1$  
	\[
		\E \left[ \sup_{n-1 \leq t \leq n} \, \sup_{|x|\leq L(t)} \left|v^{(M)}(t, x) \right|^2 \right] 
		\lesssim  \e^{-n}. 
	\] 
	By Chebyshev's inequality, we have 
	\begin{align}\label{eq:An2} 
		A_n^{(2)} &\leq \P\left\{ \sup_{n-1 \leq t\leq n} \sup_{|x|\leq L(t)}
			v^{(M)}(t\,, x) \geq  \exp\left[ - 2(\eta_1+\eta_2) \, n^{1/3}\right] \right\} \\\notag
		&\leq  \E \left[ \sup_{n-1 \leq t \leq n} \, \sup_{|x|\leq L(t)} 
			\left|v^{(M)}(t\,, x) \right|^2 \right] \,  \exp\left[ 4(\eta_1+\eta_2) \, n^{1/3}\right]
			\leq  \exp\left( -n + 4(\eta_1+\eta_2) \, n^{1/3} \right). 
	\end{align} 
	 We  combine  \eqref{eq:An1} and \eqref{eq:An2} to see that 
	 \[  
	 	\sum_{n=1}^\infty \P\left\{\sup_{n-1\leq t\leq n} \, 
		\sup_{|x| \leq L(t)} \, \frac{u(t\,, x)}{\e^{-\eta_2 t^{1/3}}} \geq \rho \right\}< \infty.
	\]
	An appeal to the Borel-Cantelli lemma completes the proof. 
\end{proof}

\begin{spacing}{.1}

\bigskip

\noindent\textbf{Davar Khoshnevisan.}\\ 
	Department of Mathematics, University of Utah, Salt Lake City, UT 84112-0090, USA, 
	\texttt{davar@math.utah.edu}\\[1em]
	
\noindent\textbf{Kunwoo Kim.}\\
	Department of Mathematics, Pohang University of Science and Technology (POSTECH), 
	Pohang, Gyeongbuk, Korea 37673, \texttt{kunwoo@postech.ac.kr}\\[1em]

\noindent\textbf{Carl Mueller.}\\
	Department of Mathematics, University of Rochester, Rochester, NY 14627, USA, 
	\texttt{carl.e.mueller@rochester.edu}
\end{spacing}

\end{document}